\documentclass[12pt]{amsart}
\usepackage{amsmath,amsfonts,amssymb,amsthm,epsfig}
\usepackage{latexsym}
\usepackage{amssymb}

\usepackage{hyperref}

\usepackage{setspace}

\usepackage[margin=1.5in]{geometry}

\def\({\left(}
\def\){\right)}
\def\[{\left[}
\def\]{\right]}
\def\<{\left\langle}
\def\>{\right\rangle}

\def\tensor{\otimes}

\def\qed{\hfill\Box\smallskip}

\newtheorem{theorem}{Theorem}[section]
\newtheorem{lemma}[theorem]{Lemma}

\newtheorem{corollary}[theorem]{Corollary}
\newtheorem{proposition}[theorem]{Proposition}

\newtheorem{remark}[theorem]{Remark}

\DeclareMathOperator{\arccosh}{arccosh}

\DeclareMathOperator{\sign}{sign}

\DeclareMathOperator{\supp}{supp}
\DeclareMathOperator{\Var}{Var}
\DeclareMathOperator{\dist}{dist}

\DeclareMathOperator{\Pois}{Pois}
\DeclareMathOperator{\Pl}{Pl}

\begin{document}

\title[Entropy of Schur--Weyl Measures]
{Entropy of Schur--Weyl Measures}

\author[S. Mkrtchyan]{Sevak Mkrtchyan}
\email{sevakm@andrew.cmu.edu}
\address{Carnegie Mellon University, Department of Mathematical Sciences\\ Pittsburgh, PA}
%\keywords{Asymptotic representation theory; Schur--Weyl duality; Plancherel-type measures; Entropy.}

\begin{abstract}
Relative dimensions of isotypic components of $N$--th order tensor representations of the symmetric group on $n$ letters give a Plancherel--type measure on the space of Young diagrams with $n$ cells and at most $N$ rows. It was conjectured by G. Olshanski that dimensions of isotypic components of tensor representations of finite symmetric groups, after appropriate normalization, converge to a constant with respect to this family of Plancherel--type measures in the limit when $\frac N{\sqrt{n}}$ converges to a constant.
The main result of the paper is the proof of this conjecture.
\end{abstract}


\maketitle

\tableofcontents

%\newpage

\section{Introduction}
Let $N$ and $n$ be two positive integers, let $S_n$ be the symmetric group on $n$ letters and let $\mathbb{Y}^n$ be the set of Young diagrams with $n$ cells. The finite dimensional irreducible representations of $S_n$ are parametrized by the set $\mathbb{Y}^n$. Given $\lambda\in\mathbb{Y}^n$ let $V_\lambda$ be the irreducible representation of $S_n$ corresponding to the Young diagram $\lambda$ and denote $\dim\lambda=\dim V_\lambda$. 

The $N$-th order tensor representation of $S_n$ is the action of $S_n$ on the tensor product space $(\mathbb{C}^N)^{\otimes n}$ by permuting the factors in the tensor product. We are interested in isotypic components of these representations.

If $V$ is an irreducible subrepresentation of a representation $U$ of a finite group, the isotypic component of $U$ corresponding to $V$ is defined to be the sum of all subrepresentations of $U$ which are isomorphic to $V$. It is easy to show that $U$ decomposes uniquely into a direct sum of its isotypic components. 

Let $\mathbb{Y}_N^n$ denote the set of Young diagrams with $n$ cells and at most $N$ rows. It follows from Schur--Weyl duality \cite{W,FultonHarris} between the symmetric group $S_n$ and the general linear group $GL(N,\mathbb{C})$ that the irreducible representations of $S_n$ which are subrepresentations of the representation $(\mathbb{C}^N)^{\otimes n}$ are exactly the ones which correspond to Young diagrams in the set $\mathbb{Y}_N^n$. Given $\lambda\in\mathbb{Y}_N^n$ let $E_\lambda$ denote the isotypic component of $(\mathbb{C}^N)^{\otimes n}$ corresponding to $V_\lambda$. Decomposing $(\mathbb{C}^N)^{\otimes n}$ into a direct sum of its isotypic components and looking at dimensions, we obtain
\begin{equation*}
N^n=\sum_{\lambda\in\mathbb{Y}_N^n}\dim E_\lambda.
\end{equation*}
Introduce a probability measure on $\mathbb{Y}_N^n$ given by relative dimensions of the corresponding isotypic components:
\begin{equation*}
\mathbb{P}_N^n(\lambda)=\frac{\dim E_\lambda }{N^n}.
\end{equation*}
We will call the measures $\mathbb{P}_N^n(\lambda)$ Schur--Weyl measures.

The main result of this paper is the following theorem on the asymptotics of Schur--Weyl measures, which was conjectured to be true by G. Olshanski:
\begin{theorem}
\label{thm:main}
For any $c>0$, $c\neq 1$ there exists a positive number $H_c$ such that for any $\varepsilon>0$ we have
\begin{equation*}
%\label{eq:main}
\lim_{\substack{n\rightarrow \infty\\N\rightarrow \infty\\\frac{\sqrt{n}}N\rightarrow c}}\mathbb{P}_N^n\left\{\lambda\in\mathbb{Y}^n_N:\left|-\frac{1}{\sqrt{n}}\ln\frac{\dim E_\lambda}{N^n}-H_c\right|<\varepsilon\right\}=1.
\end{equation*}
\end{theorem}

We obtain an explicit, albeit quite complicated formula \eqref{eq:Hc} for the constants $H_c$.

\subsection{Entropy of the Plancherel measure}
A major inspiration for this paper is a theorem of A. Bufetov on the entropy of the Plancherel measure. The Plancherel measure is the measure on $\mathbb{Y}^n$ defined by
$$
\Pl^n(\lambda)=\frac{(\dim\lambda)^2}{n!}.
$$
The measure $\mathbb{P}^n_N(\lambda)$ can be thought of as an analog of the Plancherel measure for the tensor representations of $S_n$ since in view of Burnside's theorem $\Pl^n(\lambda)$ can be interpreted as the relative dimension of the isotypic component of the regular representation of $S_n$ corresponding to $V_\lambda$. The measure $\mathbb{P}^n_N(\lambda)$ can also be thought of as a deformation of the Plancherel measure, since for fixed $n$, the measures $\mathbb{P}_N^n$ converge pointwise to the Plancherel measure when $N\rightarrow\infty$ (see, for example, \cite[Section 3]{OlshNotes}).

The theorem of A. Bufetov, which was conjectured by Vershik and Kerov, states:
\begin{theorem}[Theorem 1.1, \cite{Bu}]
\label{thm:Buf}
There exists a positive constant $H$ such that for any $\varepsilon>0$ we have
\begin{equation*}
\lim_{n\rightarrow \infty}\Pl^n\left\{\lambda\in\mathbb{Y}^n:\left|-\frac{1}{\sqrt{n}}\ln\frac{(\dim \lambda)^2}{n!}-H\right|<\varepsilon\right\}=1.
\end{equation*}
\end{theorem}

By analogy to the Shannon-McMillan-Breiman Theorem, Vershik and Kerov have suggested to call the constant $H$ the entropy of the Plancherel measure. See \cite{Bu} for details. By the same analogy, $H_c$ can be thought of as the entropy of the family of measures $\mathbb{P}_N^{\lfloor c^2N^2\rfloor}$.

\subsection{Outline of the paper}
It was proven by P. Biane \cite{Biane2001} that appropriately scaled boundaries of random Young diagrams sampled from $\mathbb{Y}^n_N$ according to the Schur--Weyl measures converge to a limit shape in the limit $n\rightarrow\infty$, $N\rightarrow\infty$ and $\frac{\sqrt{n}}{N}\rightarrow c$ (Theorem \ref{thm:Biane}). An integral formula for the logarithm of the Schur--Weyl measure $\mathbb{P}_N^n(\lambda)$ in terms of the hook lengths and contents of $\lambda$ and the deviation of the boundary of $\lambda$ from the limit shape was obtained in \cite{M2}. In addition, it was shown in \cite{M2} that the limit shape found by Biane is the unique minimizer of this integral, and the quadratic variation was calculated. The starting point of the proof of Theorem \ref{thm:main} is this variational formula (Proposition \ref{prop:intFormMeas}). Section \ref{sec:background} provides the necessary background.

To study the limit of the variational formula it is necessary to understand the local statistical properties of the boundary of Young diagrams under the Schur--Weyl measures. Toward this end, since it is easier to deal with, we first study the local statistics under the Poissonization of the Schur--Weyl measures. The first step of the proof is to show that the Poissonization of the measures $\mathbb{P}_N^n$ with respect to $n$ are Plancherel--type measures associated with certain extreme characters of the infinite dimensional unitary group (Lemma \ref{lem:poissonization}). Borodin and Kuan have proven that these Plancherel--type measures are determinantal point processes, have obtained a contour--integral representation of the correlation kernel and have found limits of the process in various regimes. In Section \ref{sec:poissonization} we present the proof that in the case which is of relevance to this paper this determinantal process converges to the discrete sine--process, and using the depoissonization 
technique of Borodin, Okounkov and Olshanski \cite{BOO} show that in the limit $N\rightarrow\infty$ the local behavior of the boundary of Young diagrams under the Schur--Weyl measures is characterized by the discrete sine--kernel (Proposition \ref{prop:localStatOfPnN}). We also show in Section \ref{sec:poissonization} that the probability of Young diagrams which extend beyond the limit shape at either edge by at least $N^\delta,\ \delta>\frac 13,$ is exponentially small. This statement for the right edge is an immediate corollary of \cite[Theorem 1.7]{JohDisc2001}.

The next step is to obtain upper bounds for the decay of correlations of the boundary of random Young diagrams. Since the contour--integral formula of Borodin and Kuan is not very suitable for such estimates, using a method of A. Okounkov \cite{OkNASA2001} we obtain a different representation of the correlation kernel and use it to obtain various bounds for the correlation kernel of the poissonized measures (Section \ref{sec:boundForCorrKer}). We use these estimates to obtain upper bounds on the decay of correlations (Section \ref{subsec:decayOfCorr}).

We use the bounds on the decay of correlations to show in Section \ref{sec:mainProof} that the weighted sum of the indicator functions of the presence of a local pattern on the boundary of a Young diagram converges to a constant with respect to the Schur--Weyl measures. This allows us to show that all the terms in the variational formula for $\mathbb{P}^n_N(\lambda)$ which can be characterized in terms of short-range patterns converge to constants. 

In Section \ref{sec:Tail} we show that the terms which correspond to long-range interactions converge to $0$ with respect to the Schur--Weyl measures. 

\subsection{Acknowledgements}
I am deeply grateful to Alexander Bufetov for suggesting this problem to me and for numerous useful discussions on the subject. I am very grateful to Grigori Olshanski for helpful discussions. I am very grateful to Alexei Borodin for pointing out the connection to \cite{JohDisc2001}. I am also very grateful to Philippe Biane for bringing \cite{BOPlancherelType} to my attention.

\section{Background}
\label{sec:background}

\subsection{The limit shape of Young diagrams with respect to Schur--Weyl measures}
\label{sec:LimitShapes}

Represent $\lambda=(\lambda_1\geq\lambda_2\geq\ldots\lambda_N)\in\mathbb{Y}^n_N$ where $\lambda_i\in\mathbb{N}$ and $\sum\lambda_i=n$ by its diagram as shown in Figure \ref{fig:Diagram}. The longest row consists of $\lambda_1$ squares of size $1$, the next longest one of $\lambda_2$ such squares, and so on. Note that for $\lambda\in\mathbb{Y}^n_N$ the integer $N$ is not encoded in the diagram of $\lambda$.
\begin{figure}[t]
\centering
\includegraphics[width=5cm]{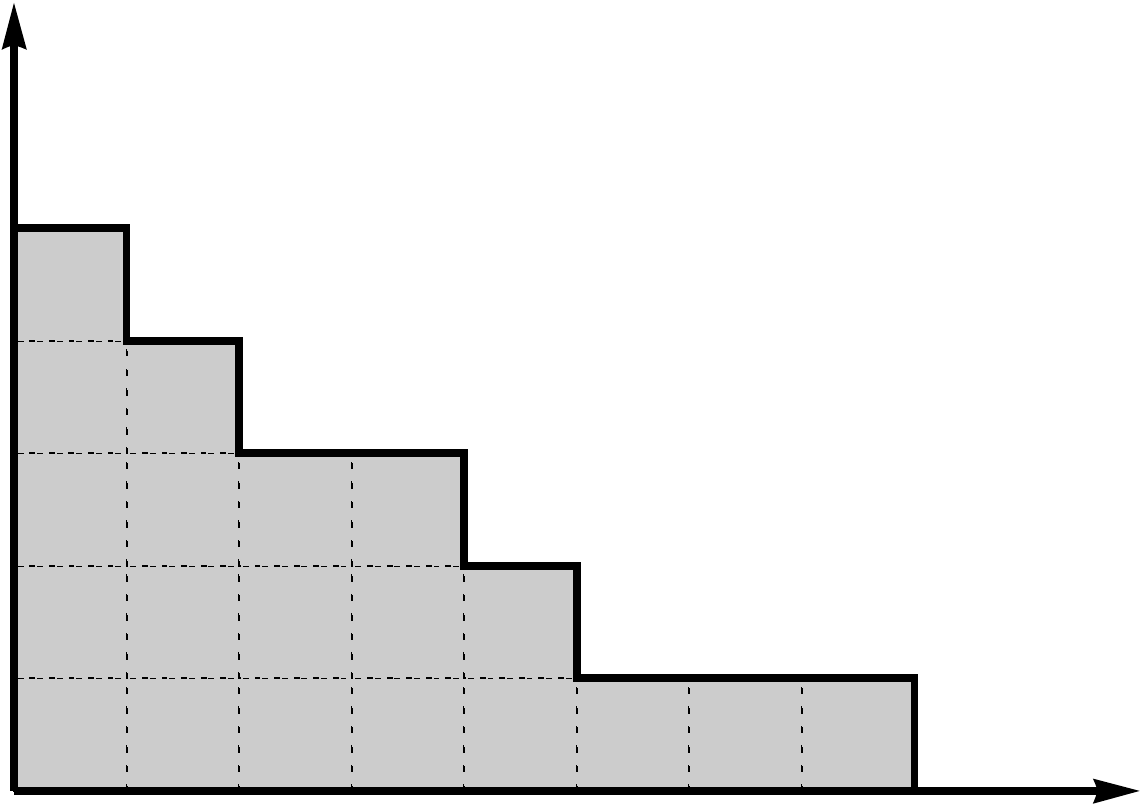}
\caption{\label{fig:Diagram} The Young diagram $\lambda=(8,5,4,2,1,0)\in\mathbb{Y}_{6}^{20}$.}
\end{figure}

Scale down the diagram by $\sqrt{\frac n2}$ in both directions so that the diagram has area $2$ and rotate the scaled diagram by $\frac \pi 4$ radians as in Figure \ref{fig:RotatedDiagram}. Let $L_\lambda(x)$ be the function giving the top boundary of the rotated diagram. Notice that $L_\lambda(x)$ is a piecewise linear function of slopes $\pm 1$ and that $L_\lambda(x)=|x|$ for $x\gg 1$ and $x\leq-\frac{N}{\sqrt{n}}$. 

\begin{figure}[ht]
\includegraphics[width=7cm]{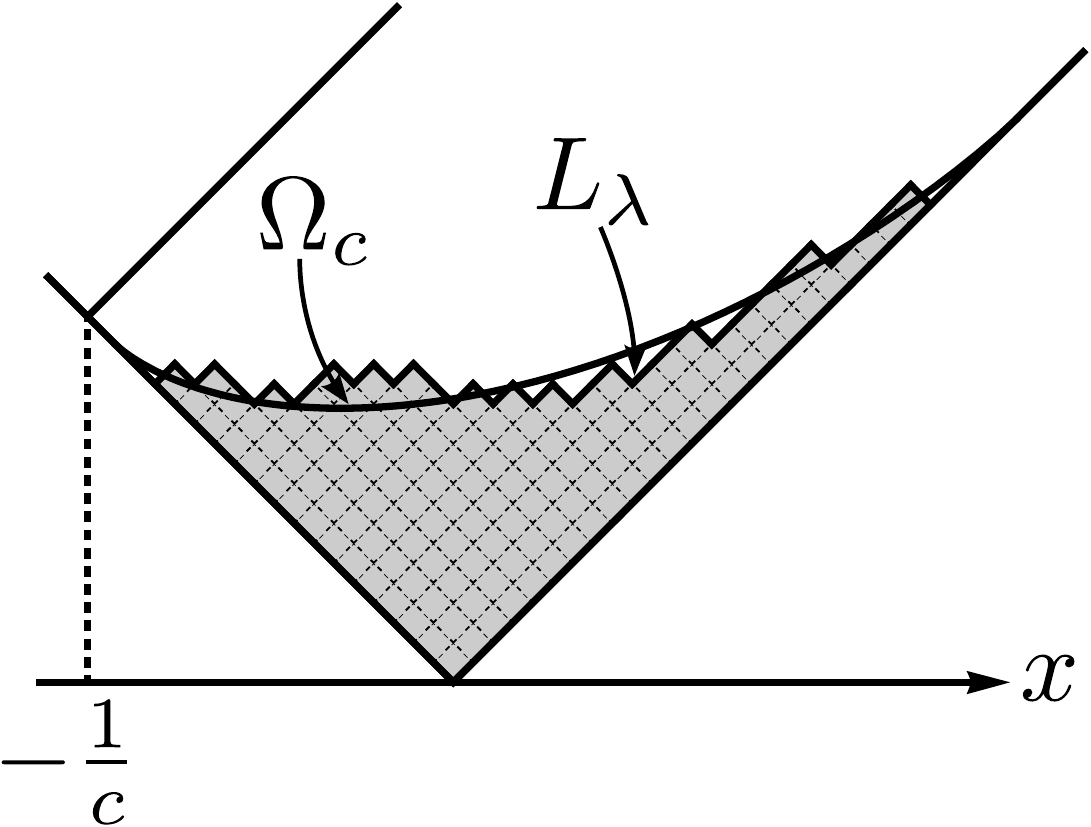}
\caption{\label{fig:RotatedDiagram}A rotated scaled Young diagram.}
\end{figure}

P. Biane \cite{Biane2001} has proven that in the limit $n\rightarrow\infty$, $\sqrt{n}/N\rightarrow c$ the boundary of a random scaled Young diagram sampled from the measure $\mathbb{P}_N^n$ converges in measure to a limit shape. The limit shape $\Omega_c(s)$ is described in the following way. For $x\in[c-2,c+2]$,
% \begin{equation*}
% \Omega_c(x)=\frac{1}{\pi} \(2x \arcsin\(\frac{x + c}{2 \sqrt{1 + x c}}\)
% 		+\frac{2}{c} \arccos\(\frac{2 + x c - c^2}{2 \sqrt{1 + x c}}\)+ \sqrt{4 - (x - c)^2}\),
% \end{equation*}
\begin{multline*}
\Omega_c(x)=\frac{1}{\pi} \(2x \arcsin\(\frac{x + c}{2 \sqrt{1 + x c}}\)
		\right.\\\left.
		+\frac{2}{c} \arccos\(\frac{2 + x c - c^2}{2 \sqrt{1 + x c}}\)+ \sqrt{4 - (x - c)^2}\),
\end{multline*}
otherwise
\begin{equation*}
	\Omega_c(x)=
	\begin{cases}
		|x|,&0<c\leq 1\text{ and }x\notin[c-2,c+2]
		\\|x|,&1<c\text{ and }x\notin[-\frac 1{c},c+2]
		\\x+\frac 2c ,&1<c\text{ and }x\in[-\frac 1{c},c-2]
	\end{cases}.
\end{equation*}

The precise formulation of Biane's theorem is the following law of large numbers for the measures $\mathbb{P}_N^n$.

\begin{theorem}[Theorem 3, \cite{Biane2001}]
\label{thm:Biane}
Let $N=N(n)$ be such that 
\begin{equation*}
\lim_{n\rightarrow\infty}\frac{\sqrt{n}}{N(n)}=c\geq0.
\end{equation*}
For any fixed $\varepsilon>0$ we have
\begin{equation*}
\lim_{n\rightarrow\infty}\mathbb{P}_N^n\{\lambda\in\mathbb{Y}^n_N:  \forall x\in\mathbb{R}, |L_\lambda(x)-\Omega_c(x)|<\varepsilon\}=1.
\end{equation*}
\end{theorem}

\begin{figure}[t]
\includegraphics[width=12cm]{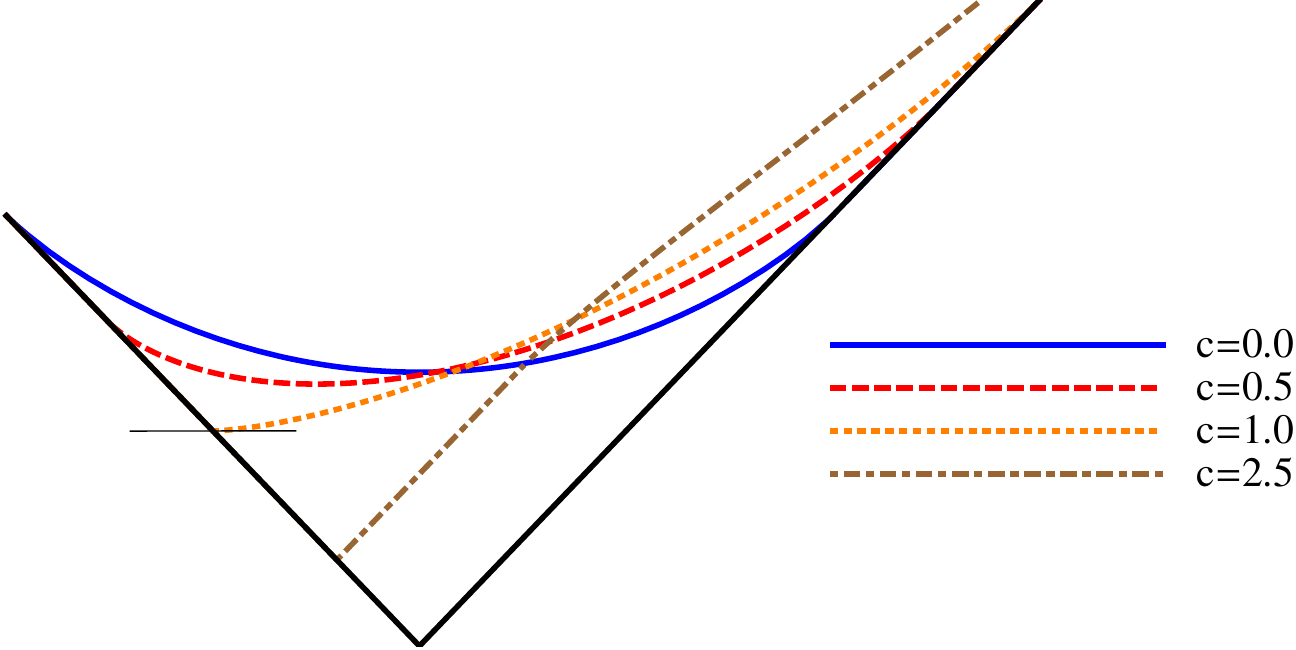}
\caption{\label{fig:OmegaC}Graphs of $\Omega_c(x)$ for $c=0,0.5,1,2.5$.}
\end{figure}

Figure \ref{fig:OmegaC} gives graphs of $\Omega_c(x)$ for several values of $c$. For every $c$ the graph of the function $\Omega_c(x)$ intersects the graph of $|x|$ at two points. All the intersections are tangential except the intersections on the left side for $c\geq 1$. At the left intersection point $\Omega_1(x)$ has slope $0$ from the right, while $\Omega_c(x)$ for $c>1$ has slope $1$ from the right.

{\it Note}: We prove Theorem \ref{thm:main} only in the case $c\neq 1$. The case $c=1$ cannot be treated together with the other cases, because the nature of the fluctuations of $L_\lambda$ near the left intersection point of the graph of $\Omega_1(x)$ with the graph of $|x|$ is different from the other cases. The main reason the nature of the fluctuations changes is the transversal intersection of the nonlinear section of the limit shape with the linear section as indicated in Figure \ref{fig:OmegaC}. The nature of fluctuations near this intersection point has been studied by Borodin and Olshanski \cite{BOPlancherelType}.

Notice that $\Omega_c(x)$ has a rather simple derivative:
\begin{equation}
\label{eq:Omega'}
\Omega_c'(x)=
	\begin{cases}
		\frac 2\pi \arcsin\left(\frac{c+x}{2\sqrt{1+xc}}\right),&x\in[c-2,c+2]
		\\1,&x>c+2\text{, or }1<c\text{ and } x\in[-\frac 1{c},c-2]
		\\-1,&\text{ otherwise }
	\end{cases}.
\end{equation}

The limit shape $\Omega_c(x)$ is a continuous deformation (depending on $c$) of the limit shape of random scaled Young diagrams sampled according to the Plancherel measure, which was found independently and simultaneously by Vershik and Kerov \cite{VK77}, and Logan and Shepp \cite{LoganShepp}. The Vershik-Kerov-Logan-Shepp limit shape is obtained when $c=0$.

\subsection{A variational formula for the measures}
Let $i$ index the rows and $j$ the columns of a Young diagram. For the cell at position $(i,j)$ in a Young diagram $\lambda$ define the numbers $p_{i,j}\in\mathbb{Z}+\frac 12$ and $q_{i,j}\in\mathbb{Z}+\frac 12$ to be $\frac 12$ plus the number of cells to the right of and respectively above the cell as shown in Figure \ref{fig:Frob}. Define the hook length of the cell at position $(i,j)$ to be $h_{i,j}=p_{i,j}+q_{i,j}$ and the content to be $c_{i,j}=j-i$.

\begin{figure}[t]
\includegraphics[width=5cm]{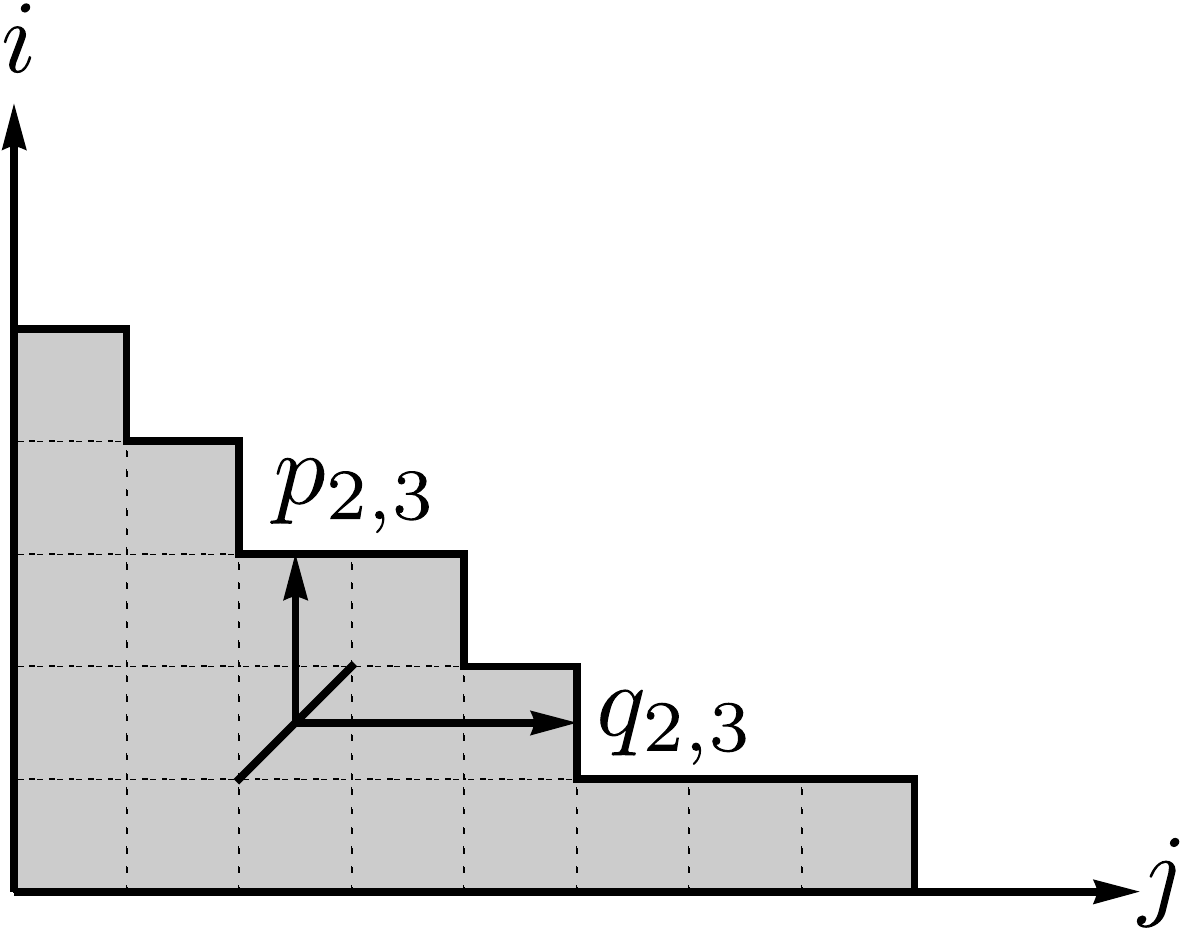}
\caption{\label{fig:Frob} $p_{2,3}=1.5$ and $q_{2,3}=2.5$ for the Young diagram $\lambda=(8,5,4,2,1)$.}
\end{figure}

For a statement $S$ denote
\begin{equation*}
\delta_S=
\begin{cases}
1,&S\text{ is true}\\
0,&S\text{ is false}
\end{cases}.
\end{equation*}

The following variational formula for the measures $\mathbb{P}_N^n$ was obtained in \cite{M2}.
\begin{proposition}[Propositions 2.1 and 3.1,\cite{M2}]
\label{prop:intFormMeas}
Let $c=c_{n,N}=\frac{\sqrt{n}}{N}>0$. We have
\begin{equation}
\label{eq:intFormMeas}
-\frac{\ln \mathbb{P}_N^n(\lambda)}{\sqrt{n}} = \frac {\sqrt{n}}8 \|f_\lambda\|_{\frac 12}^2+\frac{\sqrt{n}}{2}\int_{|x-c|>2}G_c (x)f_\lambda(x)\ dx+\hat{\theta}(\lambda)-\hat{\rho}(\lambda)-\varepsilon_n,
\end{equation}
where $f_\lambda(x)=L_\lambda(x)-\Omega_c(x)$,
\begin{equation*}
\|f\|_{\frac 12}^2=\iint{\(\frac{f(s)-f(t)}{s-t}\)^2}\ ds\ dt
\end{equation*}
is the $\frac 12$--Sobolev norm in the space of piecewise-smooth functions,
% \begin{equation*}
% G_c(x)=\delta_{|x-c|>2}\left(
% \arccosh\left|\frac{x-c}2\right|+\sign(1-c)\arccosh\left|\frac{3c-c^3+(1+c^2)x}{2(1+cx)}\right|\right),
% \end{equation*}
\begin{multline*}
G_c(x)=\\\delta_{|x-c|>2}\left(
\arccosh\left|\frac{x-c}2\right|+\sign(1-c)\arccosh\left|\frac{3c-c^3+(1+c^2)x}{2(1+cx)}\right|\right),
\end{multline*}
%when $|x-c|>2$ and $G_c(x)=0$ otherwise,
\begin{equation*}
\hat{\theta}(\lambda)=\frac {1}{\sqrt{n}}\sum_{i,j}\mathfrak{m}(h_{i,j}),
\end{equation*}
\begin{equation*}
\hat{\rho}(\lambda)=\frac{1}{2\sqrt{n}}\sum_{i,j}\mathfrak{m}(N+c_{i,j}),
\end{equation*}
\begin{equation*}
\mathfrak{m}(x)=\sum_{k=1}^\infty \frac{1}{k(k+1)(2k+1)}\frac 1{x^{2k}},
\end{equation*}
and $\varepsilon_n=o\left(\frac{\ln n}{\sqrt{n}}\right)$ is independent of $\lambda$.
The sums in $\hat{\theta}$ and $\hat{\rho}$ range over all cells of $\lambda$.
\end{proposition}

Using the varrational formula \eqref{eq:intFormMeas} it is not very hard to prove that the random variables $\frac{1}{\sqrt{n}}\ln\frac{\dim E_\lambda}{N^n}$ are {\it bounded} in measure with respect to $\mathbb{P}^n_N$ \cite{M2}:
\begin{theorem}[Theorem 1.2,\cite{M2}]
\label{thm:typDimBounds}
For any $c>0$ there exist positive numbers $\alpha_c$ and $\beta$ such that if $$\lim\limits_{n\rightarrow \infty}\frac {\sqrt{n}}{N}=c,$$ then
\begin{equation}
\lim_{n\rightarrow \infty}\mathbb{P}_N^n\left\{\lambda:\alpha_c<-\frac{1}{\sqrt{n}}\ln\frac{\dim E_\lambda}{N^n}<\beta\right\}=1.
\end{equation}
\end{theorem}
In contrast, Theorem \ref{thm:main} states that the random variables $\frac{1}{\sqrt{n}}\ln\frac{\dim E_\lambda}{N^n}$ {\it converge} to constants with respect to $\mathbb{P}^n_N$. It was proven in \cite{M2} that the quantities $\frac{1}{\sqrt{n}}\ln\frac{\max_{\lambda\in\mathbb{Y}_N^n}\{\dim E_\lambda\}}{N^n}$ are also bounded.
\begin{theorem}[Theorem 1.1,\cite{M2}]
\label{thm:maxDimBounds}
For any $c>0$ there exist positive numbers $\alpha_c$ and $\beta$ such that for large enough $n\in\mathbb{N}$ and for any $N\in\mathbb{N}$, if $c>\frac{\sqrt{n}}{N}$, then
\begin{equation}
\alpha_c<-\frac{1}{\sqrt{n}}\ln\frac{\max_{\lambda\in\mathbb{Y}_N^n}\{\dim E_\lambda\}}{N^n}<\beta.
\end{equation}
\end{theorem}

Analogous results to Theorems \ref{thm:typDimBounds} and \ref{thm:maxDimBounds} for the Plancherel measures were obtained by Vershik and Kerov in 1985 \cite{VK85}. Numerical simulations by Vershik and Pavlov \cite{VP} suggest that for the Plancherel measures the typical dimensions converge in measure (Theorem \ref{thm:Buf} by A. Bufetov). However, their simulations suggest that perhaps no such convergence holds for the maximal dimensions.

\subsection{Plancherel type measures for the infinite-dimensional unitary group}
As mentioned in the Introduction, we will need to study the poissonization of the Schur--Weyl measures. The poissonized measures are closely related to measures on signatures of length $N$ corresponding to certain extreme characters of the infinite dimensional unitary group, which we now introduce.

Let $U(N)$ denote the group of all $N\times N$ unitary matrices. There is a natural embedding of $U(N)$ into $U(N+1)$ given by 
\begin{equation*}
U(N)\ni U\mapsto
\left(\begin{array}{cc}
U&0\\
0&1
\end{array}\right)\in U(N+1).
\end{equation*}
Define the infinite dimensional unitary group $U(\infty)$ to be $$U(\infty)=\bigcup_{N=1}^\infty U(N).$$

Let $\mathbb{GT}_N$ be the set of signatures of length $N$, i.e. the set of sequences $\lambda$ of $N$ nonnegative nonincreasing integers: $\lambda=(\lambda_1\geq \lambda_2 \geq \ldots \geq \lambda_N),\ \lambda_i\in\mathbb{Z}$. It is well known that the irreducible highest--weight representations of $U(N)$ are parametrized by the set $\mathbb{GT}_N$. For $\lambda\in\mathbb{GT}_N$ let $W_\lambda$ denote the irreducible representation of $U(N)$ with highest--weight $\lambda$, and let $\chi^\lambda$ and $\dim_N \lambda$ be respectively the character and dimension of $W_\lambda$. Note that $\chi^\lambda(e)=\dim_N \lambda$, where $e$ is the identity. Define the normalized character $\tilde{\chi}^\lambda$ as $\tilde{\chi}^\lambda=\frac{\chi^\lambda}{\dim_N \lambda}$.

The notion of a normalized character can be generalized to groups such as $U(\infty)$. A normalized character of $U(\infty)$ is a positive-definite continuous function $\chi$ which is invariant under conjugation and satisfies the condition $\chi(e)=1$. The set of normalized characters of $U(\infty)$ is a convex set and the extreme characters of $U(\infty)$ are defined to be the extreme points of this set.

Extreme characters of $U(\infty)$ can be approximated by the normalized characters of $U(N)$ when $N$ goes to infinity. Here we will present the exact statement of this result only in the specific case of interest to us. For a more general discussion of extreme characters of $U(\infty)$ and for proofs see for example \cite{VK82}, \cite{OkOlJack}, \cite{BK} or \cite{BOBoundaryGT}. 

A signature $\lambda$ can be represented by two Young diagrams $(\lambda^+,\lambda^-)$ corresponding to its positive and negative parts. If $\lambda^+=(\lambda^+_1\geq\lambda^+_2\geq\ldots\geq 0)$ and $\lambda^-=(\lambda^-_1\geq\lambda^-_2\geq\ldots\geq 0)$, then 
\begin{equation*}
\lambda=(\lambda^+_1\geq\lambda^+_2\geq\ldots\geq-\lambda^-_2\geq-\lambda^-_1).
\end{equation*} 
Let ${\lambda^\pm}'$ be the transposes of $\lambda^\pm$, i.e. the number of cells in the $i$-th row of ${\lambda^\pm}'$ is equal to the number of cells in the $i$-th column of ${\lambda^\pm}$. 

For a Young diagram $\mu$ let $|\mu|$ denote the number of boxes in $\mu$ and let $d(\mu)$ denote the number of cells on the diagonal of $\mu$. The numbers $p_i(\mu):=p_{i,i}(\mu)$ and $q_i(\mu):=q_{i,i}(\mu),\ 1\leq i\leq d(\mu)$ are called Frobenius coordinates of the Young diagram $\mu$ (see Figure \ref{fig:Frob}). They completely determine $\mu$.

%\begin{theorem}[Theorem 2.1, \cite{BK}]
\begin{theorem}[\cite{VK82}]
\label{thm:extremeChars}
For any extreme character $\chi$ of $U(\infty)$ there exists a unique set of constants $\alpha^\pm_1\geq\alpha^\pm_2\geq\dots 0$, $\beta^\pm_1\geq\beta^\pm_2\geq\dots 0$ and $\delta^\pm\geq 0$, satisfying the conditions
\begin{equation*}
\sum_{i=1}^{\infty}(\alpha_i^\pm+\beta_i^\pm)< \delta^\pm,\qquad \beta_1^++\beta_1^-\leq 1,
\end{equation*}
and such that for any sequence of signatures $\lambda(N)\in\mathbb{GT}_N$, if
\begin{equation*}
\lim_{N\rightarrow\infty}\frac{p_i(\lambda(N)^\pm)}{N}=\alpha^\pm_i,\ 
\lim_{N\rightarrow\infty}\frac{q_i(\lambda(N)^\pm)}{N}=\beta^\pm_i,\ \text{and }
\lim_{N\rightarrow\infty}\frac{|\lambda(N)^\pm|}{N}=\delta^\pm,
\end{equation*}
then the normalized characters $\tilde{\chi}^{\lambda(N)}$ approximate $\chi$. %Here $p_i$ and $q_i$ are the Frobenius coordinates of the corresponding Young diagrams.
\end{theorem}

Set
\begin{equation*}
\gamma^\pm=\delta^\pm-\sum_{i=1}^\infty(\alpha^\pm_i+\beta^\pm_i)\geq 0.
\end{equation*}

Let $\chi^{\gamma^+,\gamma^-}$ denote the characters which according to Theorem \ref{thm:extremeChars} can be approximated by $\tilde{\chi}^{\lambda(N)}$ with $\alpha^\pm_i=\beta^\pm_i=0$. In other words, $\chi^{\gamma^+,\gamma^-}$ correspond to limits of $\tilde{\chi}^{\lambda(N)}$ when the rows and columns of $\lambda^\pm(N)$ grow sublinearly in $N$ and $|\lambda^\pm(N)|$ grow as $\gamma^{\pm}N$.

D. Voiculescu \cite{Voi} gave a complete description of extreme characters of $U(\infty)$. In particular, given $U\in U(\infty)$, for $\chi^{\gamma^+,\gamma^-}$ we have
\begin{equation}
\label{eq:character}
\chi^{\gamma^+,\gamma^-}(U)=\prod_{u\in \text{Spectrum}(U)}e^{\gamma^+(u-1)+\gamma^-(u^{-1}-1)}.
\end{equation}

Given a character $\chi$ of $U(\infty)$, consider its restriction to $U(N)$. It can be decomposed into a nonnegative linear combination of irreducible, and hence normalized irreducible characters of $U(N)$. Write
\begin{equation}
\label{eq:decompMeas}
\chi|_{U(N)}=\sum_{\lambda\in\mathbb{GT}_N}\mathbb{P}^\chi_N(\lambda)\tilde{\chi}^\lambda.
\end{equation}
$\mathbb{P}^\chi_N(\lambda)$ gives a probability measure on $\mathbb{GT}_N$. Let $\mathbb{P}^{\gamma^+,\gamma^-}_N$ be the measure corresponding to the extreme character $\chi^{\gamma^+,\gamma^-}$. 

\section{Poissonization and depoissonization}
\label{sec:poissonization}

All statements that follow are proven for arbitrary $c\in(0,1)\cup(1,\infty)$, however no uniformity in $c$ is established. In particular all constants may depend on $c$, but to simplify notation this dependence will not be indicated explicitly.
\subsection{Poissonization of Schur--Weyl measures}
Recall that the Poisson distribution with rate $\mu$ is
\begin{equation*}
\Pois_\mu(n)=e^{-\mu}\frac{\mu^n}{n!}.
\end{equation*}

If $\{\mathbb{P}^n\}_{n\in\mathbb{N}}$ is a family of measures with distinct supports $\{S_n\}_{n\in\mathbb{N}}$, its poissonization with parameter $\nu$ is the measure $\Pois_{\mathbb{P},\nu}$ with support $S:=\cup_{n\in\mathbb{N}} S_n$ and defined by
\begin{equation*}
\Pois_{\mathbb{P},\nu}(x)=e^{-\nu}\frac{\nu^n}{n!}\mathbb{P}^n(x),
\end{equation*}
where $\mathbb{P}^n$ is naturally extended to $S$ by setting $\mathbb{P}^n(S\backslash S_n)=0$.

Let $\mathbb{P}_{\nu,N}$ denote poissonization of the family of measures $\mathbb{P}^n_N$ with respect to $n$. It is a one--parameter family of measures on $\bigcup_{n\in\mathbb{N}}\mathbb{Y}^n_N$ defined by
\begin{equation*}
\mathbb{P}_{\nu,N}(\lambda)=e^{-\nu}\frac{\nu^n}{n!}\mathbb{P}_N^n(\lambda)\quad\text{ if }\lambda\in\mathbb{Y}^n_N.
\end{equation*}

\begin{lemma}
\label{lem:poissonization}
The measure $\mathbb{P}_N^{\gamma^+,0}$ is the poissonization of the measure $\mathbb{P}_N^n$ with respect to $n$. The poissonization parameter is $\nu=\gamma^+ N$.
\end{lemma}
\begin{proof}
% Let $\mathbb{P}_{\nu,N}$ denote the poissonization of $\mathbb{P}_N^n$ with respect to $n$, with parameter $\nu$. We have 
% \begin{equation*}
% \mathbb{P}_{\nu,N}(\lambda)=e^{-\nu}\frac{\nu^n}{n!}\mathbb{P}_N^n(\lambda).
% \end{equation*}
We need to show that $\mathbb{P}_{\gamma^+N,N}=\mathbb{P}_N^{\gamma^+,0}$. By \eqref{eq:decompMeas} it is enough to show that 
\begin{equation*}
\chi^{\gamma^+,0}|_{U(N)}=\sum_{\lambda\in\mathbb{GT}_N}\mathbb{P}_{\gamma^+N,N}(\lambda)\frac{\chi^\lambda}{\dim_N\lambda}.
\end{equation*}
By \eqref{eq:character}, for $U\in U(N)$,
\begin{equation*}
\chi^{\gamma^+,0}(U)=e^{\gamma^+ tr U - \gamma^+ N}.
\end{equation*}
It is a consequence of Schur--Weyl duality \cite{FultonHarris} that $E_\lambda=V_\lambda\tensor W_\lambda$. Hence
\begin{equation*}
\mathbb{P}_N^n(\lambda)=\frac{\dim E_\lambda }{N^n}=\frac{\dim\lambda \dim_N\lambda}{N^n},
\end{equation*}
which implies
\begin{align*}
\sum_{\lambda\in\mathbb{GT}_N}\mathbb{P}_{\gamma^+N,N}(\lambda)\frac{\chi^\lambda(U)}{\dim_N\lambda}
&=\sum_{\lambda\in\mathbb{GT}_N}e^{-\gamma^+N}\frac{(\gamma^+N)^n}{n!}\mathbb{P}_N^n(\lambda)\frac{\chi^\lambda(U)}{\dim_N\lambda}
\\&=\sum_{\lambda\in\mathbb{GT}_N}e^{-\gamma^+N}\frac{(\gamma^+N)^n}{n!}\frac{\dim\lambda \dim_N\lambda}{N^n}\frac{\chi^\lambda(U)}{\dim_N\lambda}
\\&=\sum_{\lambda\in\mathbb{GT}_N}e^{-\gamma^+N}\frac{(\gamma^+)^n}{n!}\chi^\lambda(U) \dim\lambda.
\end{align*}
Let $\mathbb{GT}_N^n$ be the set of signatures $\lambda\in\mathbb{GT}_N$ which have only nonnegative terms and for which $\sum_i \lambda_i=n$. Note that $\mathbb{GT}_N^n$ coincides with the set $\mathbb{Y}_N^n$. We obtain
\begin{align*}
\sum_{\lambda\in\mathbb{GT}_N}\mathbb{P}_{\gamma^+N,N}(\lambda)
%&
\frac{\chi^\lambda(U)}{\dim_N\lambda}
%\\
&=\sum_{n=0}^\infty e^{-\gamma^+ N}\frac{(\gamma^+)^n}{n!}\sum_{\lambda\in\mathbb{GT}_N^n}\chi^\lambda(U) \dim\lambda
\\&=\sum_{n=0}^\infty e^{-\gamma^+ N}\frac{(\gamma^+)^n}{n!}\chi^{(\mathbb{C}^N)^{\otimes n}}(U)
\\&=\sum_{n=0}^\infty e^{-\gamma^+ N}\frac{(\gamma^+)^n}{n!}(tr U)^n
\\&=e^{-\gamma^+ N} \sum_{n=0}^\infty \frac{(\gamma^+ tr U)^n}{n!}=e^{\gamma^+ tr U - \gamma^+ N},
\end{align*}
which completes the proof.

$\qed$
\end{proof}

If certain conditions are met (see Lemma \ref{lem:depoissonization}), properties of a family of measures $\mathbb{P}^n$ when $n\rightarrow\infty$ can be obtained from analogous properties of the poissonization $\Pois_{\mathbb{P},\nu}$ of those measures when $\nu\rightarrow\infty$:
\begin{equation*}
\mathbb{P}^n(x)\approx\Pois_{\mathbb{P},\nu}(x),\quad\text{ when }\nu\approx n\gg 1.
\end{equation*}

According to Lemma \ref{lem:poissonization} the poissonization of $\mathbb{P}_N^n$ with respect to $n$ with parameter $\nu$ gives $\mathbb{P}_N^{\frac \nu N,0}$. Since we are interested in properties of $\mathbb{P}^n_N$ in the limit when $n\rightarrow\infty$ so that $\frac{\sqrt{n}}N\rightarrow c$, the relevant limit of the poissonized measures $\mathbb{P}^{\gamma^+,0}_N$ for us is when the poissonization parameter $\gamma^+N$ converges to infinity so that $\frac{\sqrt{\gamma^+N}}N\rightarrow c$, or equivalently that $\frac{\gamma^+}N\rightarrow c^2$.

\subsection{The poissonized measures as determinantal point processes}
\label{sec:PointProc}
Associate with each $\lambda\in\mathbb{GT}_N$ the point configuration 
\begin{equation*}
\mathcal{P}(\lambda):=\{\lambda_1-1,\lambda_2-2,\ldots,\lambda_N-N\}\subset\mathbb{Z}.
\end{equation*}
Under this correspondence the pushforward of $\mathbb{P}_N^{\gamma^+,0}$ is a random $N$-point process on $\mathbb{Z}$. See Figure \ref{fig:Points} for a visualization of this correspondence. Note, that since the measure $\mathbb{P}_N^{\gamma^+,0}$ is supported on Young diagrams with at most $N$ rows, we are working with configurations which are subsets of $[-N,\infty)$.

\begin{figure}[t]
\includegraphics[width=7cm]{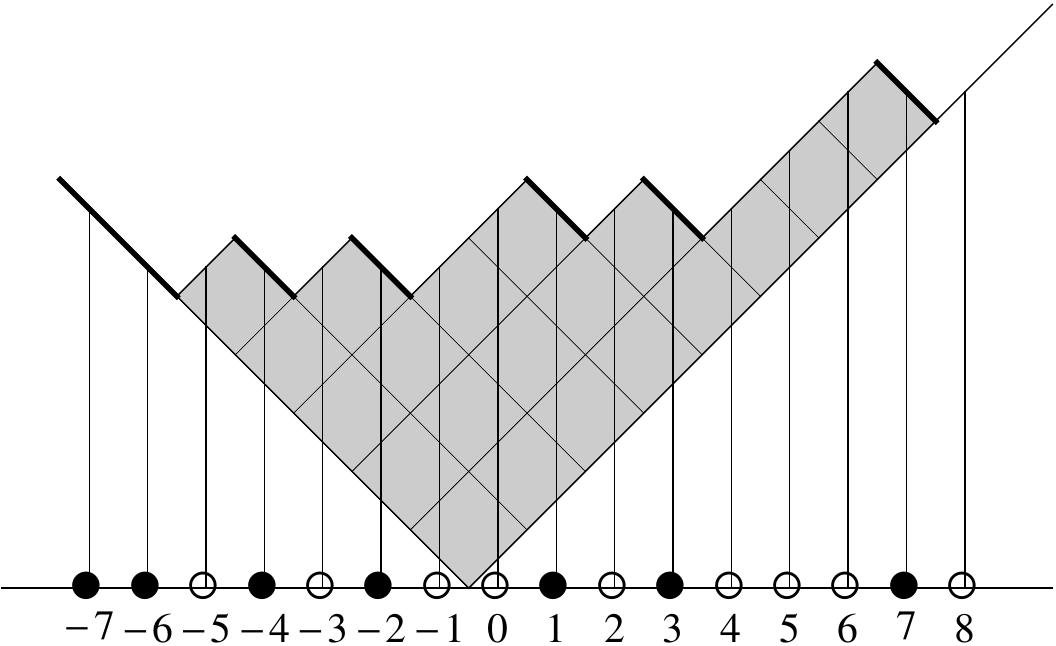}
\caption{\label{fig:Points}Black dots are points in the configuration while white dots are not.}
\end{figure}

Borodin and Kuan have proven that the point process corresponding to $\mathbb{P}_N^{\gamma^+,0}$ is determinantal. 
\begin{theorem}[Theorem 3.2, \cite{BK}]
\label{thm:intFormK}
The point process $\mathbb{P}_N^{\gamma^+,0}$ is determinantal: for arbitrary $x_1,\ldots,x_k\in\mathbb{Z}$, 
\begin{equation*}
\mathbb{P}_N^{\gamma^+,0}\left\{\lambda:\{x_1,\ldots,x_k\}\subset\mathcal{P}(\lambda)\right\}=\det[K_{N,\gamma^+}(x_i,x_j)]_{1\leq i,j\leq k}.
\end{equation*}
The correlation kernel $K_{N,\gamma^+}$ is given by
\begin{equation}
\label{eq:intFormK}
K_{N,\gamma^+}(x,y)=\frac{1}{(2\pi\mathfrak{i})^2}\oint_{|u|=r}\oint_{|w-1|=r}\frac{e^{\gamma^+u^{-1}}}{e^{\gamma^+w^{-1}}}\frac{u^{x}}{w^{1+y}}\frac{(1-u)^N}{(1-w)^N}\frac{du\ dw}{u-w},
\end{equation}
where $r$ is any constant in $(0,\frac 12)$.
\end{theorem}
{\it Note:} The theorem as stated here is a special case of the theorem of Borodin and Kuan. The theorem in \cite{BK} deals with point processes corresponding to measures on paths in the Gelfand-Tsetlin graph $\mathbb{GT}$ which arise from extreme characters of $U(\infty)$ corresponding to arbitrary parameters $(\alpha_i^\pm,\beta_i^\pm,\gamma^\pm)$.

Given an integer $k$ and a subset $X\subset\mathbb{Z}$, define
\begin{equation*}
c_k(X)=\begin{cases}1,&k\in X\\0,&k\notin X\end{cases}.
\end{equation*}
Given an integer vector $\vec{m}=\{m_1,\ldots,m_r\}$, define
\begin{equation*}
c_{\vec{m}}(X)=c_{m_1}(X)\ldots c_{m_1}(X).
\end{equation*}
For a Young diagram $\lambda$, let
\begin{equation*}
c_{\vec{m}}(\lambda)=c_{\vec{m}}(\mathcal{P}(\lambda)).
\end{equation*}

In terms of the introduced notation the statement of Theorem \ref{thm:intFormK} is equivalent to
\begin{equation*}
\mathbb{E}_{\mathbb{P}^{\gamma^+,0}_N}c_{\{x_1,\dots,x_k\}}=\det[K_{N,\gamma^+}(x_i,x_j)]_{1\leq i,j\leq k}.
\end{equation*}

Another characterization of the poissonization of the measure $\mathbb{P}_N^n$ is as the Charlier orthogonal polynomial ensemble, which was proven by K.~Johansson \cite{JohDisc2001}. Thus, the determinantal process with kernel $K_{N,\gamma^+}$ coincides with the determinantal process with the Christoffel-Darboux kernel of the Charlier ensemble. Since operators given by Christoffel-Darboux kernels are projection operators \cite{OlshDiffOpDetProc}, it follows that the operator given by $K_{N,\gamma^+}$ is also a projection operator. In particular, it follows that
\begin{equation}
\label{eq:KProj}
K_{N,\gamma^+}(x,x)=\sum_{y\in\mathbb{Z}}K_{N,\gamma^+}(x,y)K_{N,\gamma^+}(y,x)
\end{equation}
for all $x$.

\subsubsection{The discrete sine-process}
Define the discrete sine kernel to be the function
\begin{equation*}
%\label{eq:sinKer}
\mathcal{S}(l,t)=\begin{cases}\frac{\sin(lt)}{\pi l},&l\neq 0\\\frac{t}{\pi},&l=0\end{cases}.
\end{equation*}
Let $\mathbb{S}(t)$ be the measure on the power set of $\mathbb{Z}$ such that for any $m_1,\ldots,m_r\in\mathbb{Z}$, we have
\begin{equation}
\label{eq:sinProc1}
\mathbb{S}(t)\{X\subset\mathbb{Z}:m_1,\ldots,m_r\in X\}=\det[\mathcal{S}(m_i-m_j,t)]_{1\leq i,j\leq r}.
\end{equation}
The existence of such a measure follows from the general theory of determinantal point processes \cite{SoshDetP}. The measure $\mathbb{S}(t)$ is a point process on $\mathbb{Z}$ called the discrete sine-process. The condition \eqref{eq:sinProc1} can also be written as
\begin{equation*}
\label{eq:sinProc2}
\mathbb{E}_{\mathbb{S}(t)}(c_{\vec{m}})=\det[\mathcal{S}(m_i-m_j,t)]_{1\leq i,j\leq r}.
\end{equation*}
The measure $\mathbb{S}(t)$ is translation invariant: if for a constant $a$ we denote $a+\vec{m}=(a+m_1,\ldots,a+m_r)$, we have
\begin{equation*}
\mathbb{E}_{\mathbb{S}(t)}(c_{a+\vec{m}})=\mathbb{E}_{\mathbb{S}(t)}(c_{\vec{m}}).
\end{equation*}

\subsubsection{Limit of $K_{N,\gamma^+}$.}
\label{sec:limitOfK}
We show that the determinantal process given by $K_{N,\gamma^+}$ converges to the discrete sine--process when $N\rightarrow\infty,\ \frac{\gamma^+}N\rightarrow c^2$. Define the function
\begin{equation}
\label{eq:defnA}
A_{x}(z)=c^2z^{-1}+xc\ln(z)+\ln(1-z).
\end{equation}
Differentiating $A$ with respect to $z$ we obtain
\begin{equation*}
z^2(z-1)A_{x}'(z)=(1+cx)z^2-(c^2+cx)z+c^2.
\end{equation*}
If $A'_{x}(z)$ has nonreal roots, let $z^+_{x}$ be the root of $A'_{x}(z)$ such that $\Im{z^+_{x}}> 0$:
\begin{equation}
\label{eq:z+}
z^+_{x}=\frac{c^2+cx+\mathfrak{i}c\sqrt{4-(x-c)^2}}{2(1+cx)}.
\end{equation}
If $A'_{x}(z)$ has real roots, $z^+_{x}$ is the larger one. Note that $z^+_{x}\neq 0,1$. Let $z^-_{x}$ be the other root and denote $\phi_{x}=\arg(z^+_{x})$. Notice that
\begin{equation}
\label{eq:normZ+}
|z^\pm_{x}|^2=\frac{c^2}{1+cx}
\end{equation}
and
\begin{equation}
\label{eq:phi}
\phi_{x}=\arccos\left(\frac{c+x}{2\sqrt{1+cx}}\right).
\end{equation}

\begin{theorem}
\label{thm:limitOfK}
Let $\vec{x}=(x_1,\ldots,x_k)$ depend on $N$ in such a way that $x_i-x_j$ are constant and $\lim_{N\rightarrow\infty}\frac{x_j}{Nc}=x'>-\frac 1c$ for all $1\leq i,j\leq k$. If $\lim_{N\rightarrow\infty}\frac{\gamma^+}{N}=c^2$, then
% \begin{equation}
% \label{eq:limitOfK}
% \lim_{N\rightarrow\infty}\det[K_{N,\gamma^+}(x_i,x_j)]_{1\leq i,j\leq k}
% =
% \begin{cases}
% \det[\mathcal{S}(x_i-x_j,\phi_{x'})]_{1\leq i,j\leq k},&|x'-c|<2
% \\1,&x'-c\leq-2\text{ and }c<1
% \\0,&\text{ otherwise}
% \end{cases}.
% \end{equation}
\begin{multline}
\label{eq:limitOfK}
\lim_{N\rightarrow\infty}\det[K_{N,\gamma^+}(x_i,x_j)]_{1\leq i,j\leq k}
\\=
\begin{cases}
\det[\mathcal{S}(x_i-x_j,\phi_{x'})]_{1\leq i,j\leq k},&|x'-c|<2
\\1,&x'-c\leq-2\text{ and }c<1
\\0,&\text{ otherwise}
\end{cases}.
\end{multline}
\end{theorem}
{\it Note:} This is essentially a special case of Theorem 4.6 in \cite{BK}. The theorem in \cite{BK} deals with a broader family of kernels in the limit $\frac{\gamma^+}{N}\rightarrow a>0$ and $\frac{\gamma^-}{N}\rightarrow b>0$. For us $b=0$. The proof presented is an adaptation of the proof in \cite{BK} to the case $b=0$. The main reason for presenting a complete proof here is that we will need not only the result, but parts of the proof as well.

\begin{proof}
To simplify notation, in this proof we write $A(z)$ for $A_{x'}(z)$, $z^+$ for $z_{x'}^+$ and $\phi$ for $\phi_{x'}$. From \eqref{eq:intFormK} we obtain
\begin{align*}
K_{N,\gamma^+}
&
(x'cN,x'cN+l)
\\
&=\frac{1}{(2\pi\mathfrak{i})^2}\oint_{|u|=r}\oint_{|w-1|=r}\frac{e^{N(c^2u^{-1}+x'c\ln(u)+\ln(1-u)+O(\frac 1N))}}{e^{N(c^2w^{-1}+x'c\ln(w)+\ln(1-w)+O(\frac 1N))}}\frac{du\ dw}{u-w}
\\&=\frac{1}{(2\pi\mathfrak{i})^2}\oint_{|u|=r}\oint_{|w-1|=r}\frac{e^{N(A(u)-A(z^+)+O(\frac 1N))}}{e^{N(A(w)-A(z^+)+O(\frac 1N))}}\frac{du\ dw}{u-w}.
\end{align*}
We will use the saddle point method to estimate the contour integrals. For that we need to deform the contours of integration to contours $C_u$ and $C_w$, without crossing $0$, and $0$ or $1$ respectively, in such a way that
\begin{equation*}
\Re(A(z)-A(z^+))\leq 0,\ \forall z\in C_u,
\end{equation*}
and
\begin{equation*}
\Re(A(z)-A(z^+))\geq 0,\ \forall z\in C_w.
\end{equation*}
\begin{figure}[t]
\includegraphics[width=9cm]{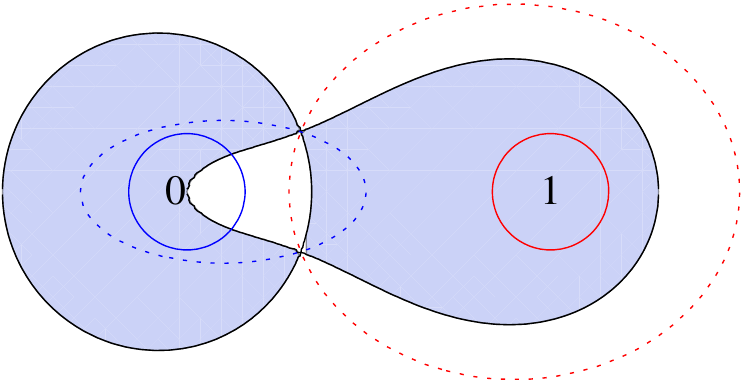}
\caption{\label{fig:contDefBulk} Deformation of contours in the bulk. The shaded region corresponds to $\Re(A(z)-A(z^+))<0$. The solid red (right) and blue (left) contours are the original contours. The dotted red and blue contours are the deformed contours.}
\end{figure}
When $z^\pm$ are not real, i.e. when $|x'-c|<2$, the contours are deformed as in Figure \ref{fig:contDefBulk}. During the deformation contours cross each other along an arc from $z^-$ to $z^+$ which crosses the real axis between $0$ and $1$, thus
% \begin{equation*}
% K_{N,\gamma^+}(x'cN,x'cN+l)
% =\frac{1}{(2\pi\mathfrak{i})^2}\oint_{C_u}\oint_{C_w}\frac{e^{N(A(u)-A(z^+)+O(\frac 1N))}}{e^{N(A(w)-A(z^+)+O(\frac 1N))}}\frac{du\ dw}{u-w}+
% \frac{1}{(2\pi\mathfrak{i})}\oint_{z^-}^{z^+}u^{-1-l}du.
% \end{equation*}
\begin{multline*}
K_{N,\gamma^+}(x'cN,x'cN+l)
\\=\frac{1}{(2\pi\mathfrak{i})^2}\oint_{C_u}\oint_{C_w}\frac{e^{N(A(u)-A(z^+)+O(\frac 1N))}}{e^{N(A(w)-A(z^+)+O(\frac 1N))}}\frac{du\ dw}{u-w}+
\frac{1}{(2\pi\mathfrak{i})}\oint_{z^-}^{z^+}u^{-1-l}du.
\end{multline*}
In the limit $N\rightarrow\infty$ the first integral goes to $0$ since the contribution to the integral from points away from the critical points is exponentially small, while at the critical points the contours $C_u$ and $C_w$ cross transversally. Thus, we obtain
\begin{equation*}
\lim_{N\rightarrow\infty}K_{N,\gamma^+}(x'cN,x'cN+l)
=\frac{1}{(2\pi\mathfrak{i})}\oint_{z^-}^{z^+}u^{-1-l}\ du.
\end{equation*}
Using \eqref{eq:normZ+}, write $z^+=\frac{c}{\sqrt{1+c x'}}e^{\mathfrak{i}\phi}$. Making the change of variable $u=\frac{c}{\sqrt{1+c x'}}e^{\mathfrak{i}\theta}$ and evaluating the remaining integral we obtain
\begin{equation*}
%\label{eq:KAsSinWithGauge}
\lim_{N\rightarrow\infty}K_{N,\gamma^+}(x'cN,x'cN+l)
=\left(\frac{c}{\sqrt{1+c x'}}\right)^{-l}\times\begin{cases}\frac{\sin(\phi l)}{\pi l},&l\neq 0\\\frac{\phi}{\pi},&l=0\end{cases}.
\end{equation*}
When taking a determinant, the gauge terms $\left(\frac{c}{\sqrt{1+c x'}}\right)^{-l}$ cancel, and we obtain \eqref{eq:limitOfK}.

The critical points $z^{\pm}$ are real when $|x'-c|\geq 2$. If $x'-c>2$, then during the deformation the contours do not cross. Thus, no residues are picked up and 
\begin{equation*}
\lim_{N\rightarrow\infty}K_{N,\gamma^+}(x'cN,x'cN+l)=0.
\end{equation*}
If $x'-c<-2$, then during the deformation one contour completely passes over the other as in Figure \ref{fig:contDefFroz}. Hence,
\begin{figure}[t]
\includegraphics[width=13cm]{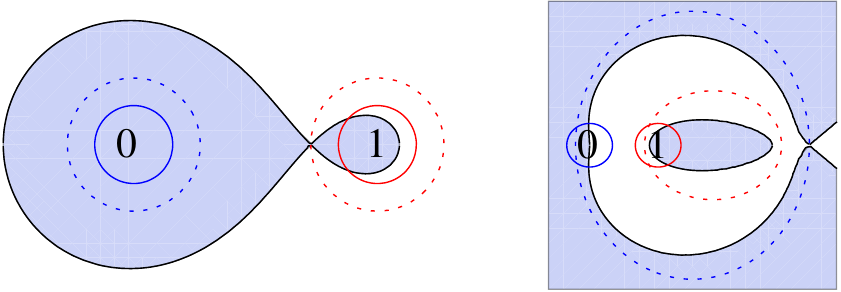}
\caption{\label{fig:contDefFroz} Deformation of contours in frozen regions. In the left figure contours do not cross. In the right figure one contour passes over the other. The shaded region corresponds to $\Re(A(z)-A(z^+))<0$. The solid red(right) and blue(left) contours are the original contours. The dotted red and blue contours are the deformed contours.}
\end{figure}
\begin{equation*}
K_{N,\gamma^+}(x'cN,x'cN+l)
=\frac{1}{(2\pi\mathfrak{i})}\oint_{\tilde{C}}u^{-1-l}\ du
\end{equation*}
for some closed contour $\tilde{C}$. When $c<1$, the contour $\tilde{C}$ winds around $0$ once and we have 
\begin{equation*}
\lim_{N\rightarrow\infty}K_{N,\gamma^+}(x'cN,x'cN+l)=\begin{cases}1,&l=0\\0,&l\neq 0\end{cases}.
\end{equation*}
When $c>1$, then $\tilde{C}$ winds around $1$, whence $K_{N,\gamma^+}(x'cN,x'cN+l)=0$. %Lastly, notice that when $c=1$, $x'-c$ cannot be less than $-2$, since $x'>\frac 1c$. 

\begin{figure}[t]
\includegraphics[width=11cm]{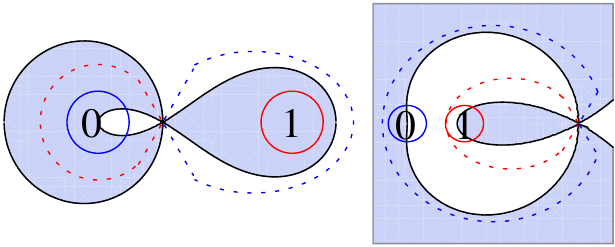}
\caption{\label{fig:contDefCrit} Deformation of contours at the transition points. In the left figure contours do not cross. In the right figure one contour passes over the other. The shaded region corresponds to $\Re(A(z)-A(z^+))<0$. The solid red(right) and blue(left) contours are the original contours. The dotted red and blue contours are the deformed contours and both are linear near the double critical point.}
\end{figure}

In the case $|x'-c|=2$, $A(z)$ has double real critical points and the contours should be deformed as shown in Figure \ref{fig:contDefCrit}. This case can be analyzed similarly by noting that the contribution from the neighborhood of the double critical point is negligible if contours are deformed as shown.
$\qed$
\end{proof}

\subsection{Depoissonization and local statistics of Schur--Weyl measures in the bulk}

\subsubsection{Depoissonization}
A lemma proven by Borodin, Okounkov and Olshanski \cite[Lemma 3.1]{BOO} allows us to pass from asymptotic properties of the poissonized measures to analogous properties of the original measures.
% **** Maybe modify to fit better with my notation****
% \begin{lemma}[Borodin, Okounkov, Olshanski, \cite{BOO}][Depoissonization Lemma]
% \label{lem:depoissonization}
% Let $0<\alpha<\frac 14$. Let $\{f_n\}$ be a sequence of entire functions
% \begin{equation*}
% f_n(z)=e^{-z}\sum_{k\geq 0}\frac{f_{nk}}{k!}z^k,\ n=1,2,\ldots
% \end{equation*}
% and assume that there exist constants $f_\infty$, $C_1$, $C_2$ and $C_3$ such that
% \begin{equation*}
% %\label{eq:depoisCond1}
% \max_{|z|=n}|f_n(z)|\leq C_1e^{C_2\sqrt{n}}
% \end{equation*}
% and
% \begin{equation*}
% %\label{eq:depoisCond2}
% \max_{|z-n|<n^{1-\alpha}}\frac{|f_n(z)-f_\infty|}{e^{C_2|z-n|/\sqrt{n}}}\leq C_3.
% \end{equation*}
% Then there exists a constant $C=C(C_1,C_2,C_3)$, such that for all $n>0$ we have
% \begin{equation*}
% |f_{nn}-f_\infty|\leq C.
% \end{equation*}
% \end{lemma}
% {\it Note:} The version of the Depoissonization Lemma as stated here is from Lemma ****?**** in \cite{Bu}. It is a slight modification of the original lemma proven in \cite{BOO}.
We present a modified version of the Depoissonization Lemma, which appeared in \cite{Bu}.
\begin{lemma}[Corollary 3.3 in \cite{Bu}][Depoissonization Lemma]
\label{lem:depoissonization}
Let $0<\alpha<\frac 14$ and let $\{f_n\}$ be a sequence of entire functions
\begin{equation*}
f_n(z)=e^{-z}\sum_{k\geq 0}\frac{f_{nk}}{k!}z^k,\quad n=1,2,\ldots.
\end{equation*}
Let $\tilde{C}>0$, and let $a_n$ be a sequence of positive numbers satisfying $|a_n|\leq \tilde{C}$. If there exist constants $C_1,C_2,g_1,g_2>0$ and $f_\infty$ such that
\begin{equation*}
%\label{eq:depoisCond1}
\max_{|z|=n}|f_n(z)|\leq C_1e^{g_1\sqrt{n}}
\end{equation*}
and
\begin{equation*}
%\label{eq:depoisCond2}
\max_{|z-n|<n^{1-\alpha}}\frac{|f_n(z)-f_\infty|}{e^{g_2|z-n|/\sqrt{n}}}\leq C_2 a_n,
\end{equation*}
then there exists a constant $C$ depending only on $C_1,C_2,g_1,g_2$ and $\tilde{C}$, such that for all $n>0$ we have
\begin{equation*}
|f_{nn}-f_\infty|\leq C a_n.
\end{equation*}
\end{lemma}

\subsubsection{Depoissonization of $\mathbb{P}^{\gamma^+,0}_N$}
For $\varepsilon\geq 0$ denote
\begin{equation*}
\mathcal{I}_N(\varepsilon)=\left\{k\in\mathbb{Z}:\left|\frac k{cN}-c\right|\leq 2-\varepsilon\right\},
\end{equation*}
and for $\delta>0$ and $K>0$ denote
\begin{equation*}
\mathcal{I}^\pm_N(K,\delta)=\left\{k\in\mathbb{Z}:\left|\frac k{cN}-c\right|\leq 2\pm K N^{\delta-1}\right\}.
\end{equation*}

\begin{lemma}
\label{lem:depoisK1}
There exist constants $C_1,C_2>0$ such that 
\begin{equation}
%\label{eq:depoisK1}
\max_{|\gamma^+|=c^2N}\left|K_{N,\gamma^+}(x,y)\right|\leq C_1 e^{C_2 N} \left(\frac 32\right)^{y-x}
\end{equation}
for all $x$ and $y$ such that $\frac x{cN},\frac y{cN}>-\frac 1c$.
\end{lemma}
{\it Note}: Henceforth, whenever studying the kernel $K_{N,\gamma^+}$ and not the measure $\mathbb{P}^{\gamma^+,0}_N$, we will allow $\gamma^+$ to be a complex parameter. In particular, this will be the case in depoissonization lemmas.

\begin{proof} Let $\tilde{\gamma}=\frac{\gamma^+}{c^2N}$ and $l=y-x$. Using the contour--integral estimation result which states that for a continuous function $f$ it holds that
\begin{equation*}
\left|\int_C f(z)dz\right|\leq \max_{z\in C}|f(z)|l(C),
\end{equation*}
where $l(C)$ is the length of the contour $C$, we obtain from \eqref{eq:intFormK} that
\begin{equation*}
|K_{N,\gamma^+}(x,x+l)|\leq r^2e^{c^2N\max\left(\Re(\tilde{\gamma}(u^{-1}-w^{-1}))\right)}\frac{r^{x}}{(1\mp r)^{1+x+l}}\frac{(1+r)^N}{r^N}\frac{1}{1-2r},
\end{equation*}
where the plus sign is chosen when $1+x+l<0$.
Since $|\tilde{\gamma}|=1$, $|u|=r$ and $|w|\geq 1-r$, it follows that 
\begin{equation*}
|\tilde{\gamma}(u^{-1}-w^{-1})|\leq \frac 1r+\frac1{1-r},
\end{equation*}
whence 
\begin{equation*}
\max\left(\Re(\tilde{\gamma}(u^{-1}-w^{-1}))\right)\leq\frac1{r(1-r)}.
\end{equation*}
Thus,
\begin{equation*}
|K_{N,\gamma^+}(x,x+l)|\leq\frac{r^2(1-r)^{-1-l}}{1-2r}e^{N\left(\frac xN\ln(r)-\ln(r)+\ln(1+r)-\frac xN\ln(1\mp r)+\frac{c^2}{r(1-r)}\right)}.
\end{equation*}
Since the coefficient of $\frac xN$ is $\ln(r)-\ln(1\mp r)<0$ and $\frac xN$ is bounded below, taking $r=\frac 13$ completes the proof.

$\qed$
\end{proof}

\begin{lemma}
\label{lem:depoisK2}
%Suppose $c\neq 1$.
For any $\delta_0>\frac 13$ and any integer $l$ there exist constants $C_1=C_1(\delta_0,l)>0$ and $C_2=C_2(\delta_0,l)>0$ such that 
\begin{equation*}
\left|K_{N,\gamma^+}(x,x+l)-\left(\frac{c}{\sqrt{1+\frac{x}{N}}}\right)^{-l}\mathcal{S}\left(l,\phi_{\frac x{cN}}\right)\right|\leq \frac{C_1 e^{C_2|\gamma^+-c^2N|}}{2cN-|x-c^2N|}
\end{equation*}
for all $\gamma^+$, all $\delta\in[\delta_0,1)$, all $N\in\mathbb{N}$ and all $x\in\mathcal{I}^-_N(1,\delta)$.
\end{lemma}
%{\it Note:} The condition $c\neq 1$ is only needed near the left edge. If $-(2-\varepsilon)\leq\frac x{cN}-c\leq 2-N^{\delta-1}$, for $\varepsilon>0$ independent of $N$, then the result holds for $c=1$ as well.

\begin{proof} Throughout the proof $C_1$ and $C_2$ will denote arbitrary constants that depend only on $\delta_0$ and $l$. In this proof the indices of $A$, $z^\pm$ and $\phi$ are $\frac x{cN}$, however, to simplify notation, we will omit those indices.

Let $\tilde{\gamma}=\gamma^+-c^2N$. For contours $S_1$ and $S_2$ define $K_{S_1,S_2}$ to be
\begin{equation}
\label{eq:intFormKNearN}
K_{S_1,S_2}
=\frac{1}{(2\pi\mathfrak{i})^2}\oint_{S_1}\oint_{S_2}
\frac{e^{N(A(u)-A(z^+))}}{e^{N(A(w)-A(z^+))}}
\frac{e^{\tilde{\gamma}u^{-1}}}{e^{\tilde{\gamma}w^{-1}}}
\frac{1}{w^{l+1}}\frac{du\ dw}{u-w}.
\end{equation}
% \begin{equation}
% \label{eq:intFormKNearN}
% K_{N,\gamma^+}(xcN,xcN+l)
% =\frac{1}{(2\pi\mathfrak{i})^2}\oint_{|u|=r}\oint_{|w-1|=r}
% \frac{e^{N(c^2u^{-1}+xc\ln(u)+\ln(1-u))}}{e^{N(c^2w^{-1}+xc\ln(w)+\ln(1-w))}}
% \frac{e^{\tilde{\gamma}u^{-1}}}{e^{\tilde{\gamma}w^{-1}}}
% \frac{1}{w^{l+1}}\frac{dudw}{u-w}.
% \end{equation}
It follows from \eqref{eq:intFormK} that
\begin{equation*}
K_{N,\gamma^+}(x,x+l)=K_{|u|=r,|w-1|=r},
\end{equation*}
where $0<r<\frac 12$. It follows from the proof of Theorem \ref{thm:limitOfK} that
\begin{equation*}
K_{N,\gamma^+}(x,x+l)-\frac{1}{(2\pi\mathfrak{i})}\oint_{z^-}^{z^+}u^{-1-l}\ du=K_{C_u,C_w}.
\end{equation*}

Let $\frac x{cN}-c=\pm(2-pN^{\delta-1})$ for some $p>0$. From \eqref{eq:defnA} and \eqref{eq:z+} we obtain
\begin{equation}
\label{eq:asympOfZ+}
z^+=\frac c{c\pm 1}+\mathfrak{i}\frac{c\sqrt{p}N^{\frac{\delta-1}2}}{(c\pm 1)^2}+O\left(N^{\delta-1}\right),
\end{equation}
\begin{equation}
\label{eq:asympOfA''}
A''(z^+)=\mp\mathfrak{i}\frac{2(c\pm 1)^3\sqrt{p}N^{\frac{\delta-1}2}}{c}+O\left(N^{\delta-1}\right),
\end{equation}
and
\begin{equation}
\label{eq:asympOfA'''}
A'''(z^+)=\mp\frac{2(c\pm 1)^5}{c^2}-\mathfrak{i}\frac{6(c\mp 2)(c\pm 1)^4\sqrt{p}N^{\frac{\delta-1}2}}{c^2}+O\left(N^{\delta-1}\right).
\end{equation}

Since $A'(z^+)=0$, Taylor's theorem implies
\begin{align*}
A(z^++e^{\mathfrak{i}\xi}t)-A(z^+)=&\frac 12 (e^{\mathfrak{i}\xi}t)^2 A''(z^+) + \frac{(e^{\mathfrak{i}\xi}t)^3}{3!}A'''(z^+)
% \\&
+\frac{(e^{\mathfrak{i}\xi}t)^4}{4!}R_3(z^++e^{\mathfrak{i}\xi}t),
\end{align*}
where
\begin{equation*}
R_3(u)=\frac 1{2\pi\mathfrak{i}}\oint\frac{A(z)}{(z-z^+)^4(z-u)}\ dz,
\end{equation*}
and the last integration is over a closed contour that contains both $z^+$ and $u$.

We can assume that the contours $C_u$ and $C_w$ are linear near the critical points $z^\pm$, i.e. that there exist $\xi,\psi\in(0,\frac{\pi}2)\cup(\frac{\pi}2,\pi)$ and $t_0>0$ such that the contours $C_u$ and $C_w$ coincide respectively with $z^\pm+e^{\pm\mathfrak{i}\xi}t$ and $z^\pm+e^{\pm\mathfrak{i}\psi}t$, when $|t|<t_0$, $t\in\mathbb{R}$. We have
\begin{equation*}
\Re(A(z^++e^{\mathfrak{i}\xi}t)-A(z^+))<0<\Re(A(z^++e^{\mathfrak{i}\psi}t)-A(z^+))
\end{equation*}
for all $0<|t|<t_0$, $t\in\mathbb{R}$. Let $\beta$ be a constant such that $\frac 13<\beta<\frac{\delta_0+1}{4}$. Since $\Im(z^+)>N^{-\beta}$, we can divide the contour $C_u$ into three sections as follows:
\begin{align}
\label{eq:contourPieces}
\nonumber C_{u,\pm}=z^{\pm}+&e^{\pm\mathfrak{i}\xi}t\text{ when }|t|<N^{-\beta}
\\&\text{ and }
\\\nonumber C'_u=C_u&\backslash (C_{u,+}\cup C_{u,-}).
\end{align}
\begin{figure}[t]
\includegraphics[width=7cm]{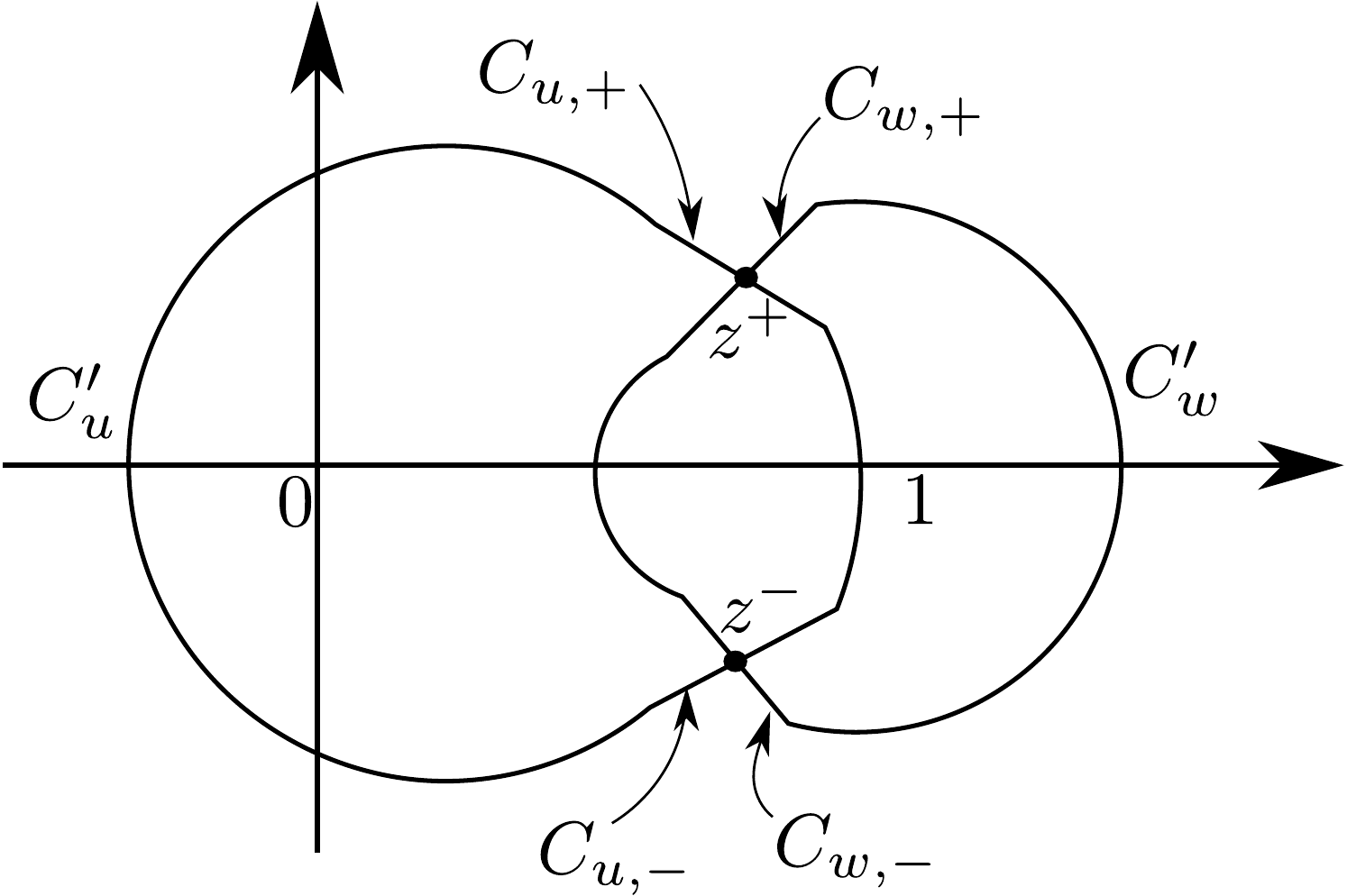}
\caption{\label{fig:ContourSections} Sections of the contours $C_u$ and $C_w$.}
\end{figure}
Similarly, divide $C_w$ into three sections $C_{w,\pm}$ and $C'_w$ (see Figure \ref{fig:ContourSections}). We estimate the contribution of each section separately. We will present the proofs of the following two estimates:
\begin{align}
\label{eq:contLarge}
|K_{C'_u,C_w}|&<\frac{C_1e^{C_2|\tilde{\gamma}|}}{2cN-|x-c^2N|},
\\\label{eq:contSm+}
|K_{C_{u,+},C_{w,+}}|&<\frac{C_1e^{C_2|\tilde{\gamma}|}}{2cN-|x-c^2N|}.
\end{align}
Estimates for the other sections can be obtained completely similarly.

We start with proving \eqref{eq:contLarge}. Since the leading term of $A''(z^+)$ is of order $N^{\frac{\delta-1}2}$, $\xi\neq\frac{\pi}2$ and $R_3(u)$ is bounded in a neighborhood of $z^+$, there exists $D_1\in(0,t_0)$ such that
\begin{equation*}
\Re(A(u)-A(z^+))
\leq
\begin{cases}
-D_2\sqrt{p}N^{\frac{\delta-1}2}t^2, &u=z^\pm+e^{\pm\mathfrak{i}\xi}t\text{ and }|t|\leq D_1
\\-D_3,&u\in C'_u\cap\{u:|u-z^\pm|>D_1\}
\end{cases}
\end{equation*}
for some positive constants $D_2$ and $D_3$.

Since $|u|$ and $|w|$ are bounded away from zero and are bounded above along the contours $C_u$ and $C_w$, it follows that 
\begin{equation*}
\left|\frac 1{w^{l+1}}\frac{e^{\tilde{\gamma}u^{-1}}}{e^{\tilde{\gamma}w^{-1}}}\right|\leq C_1e^{C_2\tilde{\gamma}}.
\end{equation*}
Since $|u-w|^{-1}<N^{\beta}$ along the contours $C'_u$ and $C_w$, and $\left|e^{N(A(w)-A(z^+))}\right|>1$ for all $w\in C_w$, we obtain
\begin{equation}
\label{eq:contLL}
|K_{C'_u\cap\{u:|u-z^\pm|>D_1\},C_w}|\leq C_1N^{\beta}e^{C_2\tilde{\gamma}} e^{-D_3 N}.
\end{equation}
For the remaining part of the contour $C'_u$ we obtain
\begin{equation}
\label{eq:contLl1}
|K_{C'_u\cap\{u:|u-z^+|<D_1\},C_w}|\leq C_1e^{C_2\tilde{\gamma}}N^\beta\int_{N^{-\beta}}^{D_1}e^{-D_2\sqrt{p}N^{\frac{\delta+1}2}t^2}\ dt.
\end{equation}
Making the change of variable $t'=\sqrt{D_2\sqrt{p}}N^{\frac{\delta+1}4} t$, we obtain 
\begin{align}
\label{eq:contLl2}
\int_{N^{-\beta}}^{D_1}e^{-D_2\sqrt{p}N^{\frac{\delta+1}2}t^2\ }dt
&=\frac {N^\beta}{\sqrt{D_2\sqrt{p}}N^{\frac{\delta+1}4}}\int_{\sqrt{D_2\sqrt{p}}N^{\frac{\delta+1}4-2\beta}}^{D_1 \sqrt{D_2\sqrt{p}}N^{\frac{\delta+1}4}}e^{-t'^2}\ dt'
\\\nonumber &<
%\frac 1{\sqrt{D_2\sqrt{p}}N^{\frac{\delta+1}4}}\int_{\sqrt{D_2\sqrt{p}}N^{\frac{\delta+1}4-\beta}}^{\infty}e^{-\sqrt{D_2\sqrt{p}}N^{\frac{\delta+1}4-\beta}t'}dt'=
\frac 1{D_2\sqrt{p} N^{\frac{\delta+1}2-\beta}}e^{-D_2\sqrt{p} N^{\frac{\delta+1}2-2\beta}}.
\end{align}
Since $\beta<\frac{\delta+1}4$, combining \eqref{eq:contLL}, \eqref{eq:contLl1} and \eqref{eq:contLl2} we obtain \eqref{eq:contLarge}.

We now move on to proving \eqref{eq:contSm+}. We will consider two cases: when $|\tilde{\gamma}|$ is large and when it is small. Let $\zeta\in(0,\beta)$ and suppose $|\tilde{\gamma}|<N^\zeta$.

Since $A'(z^+)=0$, we obtain
\begin{equation}
\label{eq:ATaylor2}
A(z^+ + e^{\mathfrak{i}\xi}t)-A(z^+)=\frac 12 e^{\mathfrak{i}2\xi}A''(z^+) t^2 + t^3\left(\frac 16 e^{\mathfrak{i}3\xi} A'''(z^+)+O(t)\right).
\end{equation}
Since $\beta>\frac 13$, it follows that $Nt^3=o(1)$, whence
\begin{equation}
\label{eq:estCubicTerm}
\frac{e^{Nt^3\left(\frac 16 e^{\mathfrak{i}3\xi} A'''(z^+)+O(t)\right)}}{e^{Ns^3\left(\frac 16 e^{\mathfrak{i}3\psi} A'''(z^+)+O(s)\right)}}=1+O(|t|^3+|s|^3)N.
\end{equation}
Since $|\tilde{\gamma}t|<N^{\zeta-\beta}$ and $\zeta-\beta<0$, we obtain
\begin{equation}
\label{eq:estGammaTerm}
\frac{e^{\tilde{\gamma}(z^++e^{\mathfrak{i}\xi}t)^{-1}}}
{e^{\tilde{\gamma}(z^++e^{\mathfrak{i}\psi}s)^{-1}}}
=\frac{e^{\tilde{\gamma}(z^+)^{-1}+O(|\tilde{\gamma}t|)}}
{e^{\tilde{\gamma}(z^+)^{-1}+O(|\tilde{\gamma}s|)}}
=1+|\tilde{\gamma}|O(|t|+|s|).
\end{equation}
Define
\begin{equation*}
\mathfrak{A}(t,s)=\frac{e^{N\frac 12 e^{\mathfrak{i}2\xi}A''(z^+) t^2}}{e^{N\frac 12 e^{\mathfrak{i}2\psi}A''(z^+)s^2}}.
\end{equation*}
Since the function
\begin{equation*}
\mathfrak{A}(t,s)\frac 1{t-e^{\mathfrak{i}(\psi-\xi)}s}
\end{equation*}
is an odd function, it follows that
\begin{equation}
\label{eq:intOddZero}
\int_{-N^{-\beta}}^{N^{-\beta}}\int_{-N^{-\beta}}^{N^{-\beta}}\mathfrak{A}(t,s)\frac {dt\ ds}{t-e^{\mathfrak{i}(\psi-\xi)}s}=0.
\end{equation}
Using \eqref{eq:contourPieces}, \eqref{eq:ATaylor2}, \eqref{eq:estCubicTerm}, \eqref{eq:estGammaTerm} and \eqref{eq:intOddZero}, and noting that
\begin{equation*}
%\label{eq:estWterm}
\frac{1}{(z^+ +e^{\mathfrak{i}\psi}s)^{l+1}}=\frac{1}{(z^+)^{l+1}}+O(s),
\end{equation*}
rewrite \eqref{eq:intFormKNearN} as
% \begin{equation*}
% K_{C_{u,+},C_{w,+}}=\int_{-N^{-\beta}}^{N^{-\beta}}\int_{-N^{-\beta}}^{N^{-\beta}}
% \mathfrak{A}(t,s)\left(O(|t|^3+|s|^3)N+|\tilde{\gamma}|O(|t|+|s|)+O(|s|)\right)\frac{dt\ ds}{t-e^{\mathfrak{i}(\psi-\xi)}s}.
% \end{equation*}
\begin{multline*}
K_{C_{u,+},C_{w,+}}=\int_{-N^{-\beta}}^{N^{-\beta}}\int_{-N^{-\beta}}^{N^{-\beta}}
\mathfrak{A}(t,s)\\\times\left(O(|t|^3+|s|^3)N+|\tilde{\gamma}|O(|t|+|s|)+O(|s|)\right)\frac{dt\ ds}{t-e^{\mathfrak{i}(\psi-\xi)}s}.
\end{multline*}
Making the change of variable
\begin{equation}
\label{eq:changeVarTS}
t'=p^{\frac 14}N^{\frac {\delta+1}4}t,\quad s'=p^{\frac 14}N^{\frac {\delta+1}4}s,
\end{equation}
and using \eqref{eq:asympOfA''} we obtain
\begin{multline*}
%\label{eq:estTotal}
|K_{C_{u,+},C_{w,+}}|\leq\frac 1{p^{\frac 12}N^{\frac{\delta+1}2}}
\int_{-p^{\frac 14}N^{\frac{\delta+1}4-\beta}}^{p^{\frac 14}N^{\frac{\delta+1}4-\beta}}\int_{-p^{\frac 14}N^{\frac{\delta+1}4-\beta}}^{p^{\frac 14}N^{\frac{\delta+1}4-\beta}}
e^{-D_4(t'^2+s'^2)}
\\\times\left(\frac{1}{p^{\frac 12}N^{\frac{\delta-1}2}}O(|t'|^3+|s'|^3)+|\tilde{\gamma}|O(|t'|+|s'|)+O(|s'|)\right)\frac{dt'\ ds'}{|t'-e^{\mathfrak{i}(\psi-\xi)}s'|},
\end{multline*}
where $D_4$ is a positive constant. Since the remaining integral is $O(N^{-\frac{\delta-1}2}+|\tilde{\gamma}|)$, we obtain
\begin{equation*}
|K_{C_{u,+},C_{w,+}}|\leq\frac {C_1}{pN^{\delta}}e^{C_2|\tilde{\gamma}|}.
\end{equation*}
This completes the proof of \eqref{eq:contSm+} when $|\tilde{\gamma}|<N^{\zeta}$.

The case $|\tilde{\gamma}|>N^\zeta$ is much simpler. From \eqref{eq:intFormKNearN} it follows that
\begin{equation}
\label{eq:largeGamma}
|K_{C_{u,+},C_{w,+}}|\leq C_1e^{C_2|\tilde{\gamma}|}
\int_{-N^{-\beta}}^{N^{-\beta}}\int_{-N^{-\beta}}^{N^{-\beta}}
|\mathfrak{A}(t,s)|\frac{dt\ ds}{|t-e^{\mathfrak{i}(\psi-\xi)}s|}.
\end{equation}
Making the change of variable \eqref{eq:changeVarTS} it is easy to see that the remaining integral is $O(1)$. Since $|\tilde{\gamma}|>N^\zeta$, \eqref{eq:contSm+} follows from \eqref{eq:largeGamma}.

$\qed$
\end{proof}

\begin{proposition}[Local statistics of $\mathbb{P}_N^n$ in the bulk]
\label{prop:localStatOfPnN}
For any $\varepsilon>0$ and any integer $L>0$, there exists a positive constant $C=C(\varepsilon,L)$ such that for all $x\in\mathcal{I}_N(\varepsilon)$, all integer vectors $\vec{l}$ satisfying $|\vec{l}|\leq L$, all $N\in\mathbb{N}$ and $n=\lfloor c^2N^2\rfloor$, we have
\begin{equation*}
%\label{eq:localStatOfPnN}
\left|\mathbb{E}_{\mathbb{P}_N^n}(c_{x+\vec{l}})-\mathbb{E}_{\mathbb{S}(\phi_{\frac x{cN}})}(c_{\vec{l}})\right|\leq\frac{C(\varepsilon,L)}{N}.
\end{equation*}
\end{proposition}
\begin{proof}
This follows by applying the depoissonization Lemma \ref{lem:depoissonization} to Theorem \ref{thm:limitOfK}. Lemmas \ref{lem:depoisK1} and \ref{lem:depoisK2} show that the necessary conditions for Lemma \ref{lem:depoissonization} to apply are satisfied.

$\qed$
\end{proof}

\subsection{Statistics near edges}
We now prove that the probability of Young diagrams which extend beyond the limit shape at either edge by a distance more than $N^{\delta}$ with $\delta>\frac 13$ are exponentially small. We will need the following lemma, which gives an estimate for $K_{N,\gamma^+}(x,x)$ near the edges.
\begin{lemma}
\label{lem:KExtremes}
For any $\delta_0>\frac 13$ there exist constants $C_1,C_2,C_3>0$ such that for all $\delta\in[\delta_0,1)$, for all $\gamma^+$, for all $N\in\mathbb{N}$, and $x\notin \mathcal{I}^+_N(1,\delta), x>-N$, we have
\begin{equation}
\label{eq:KEdgeSmC}
\left|1-K_{N,\gamma^+}(x,x)\right|\leq C_1 e^{-C_2 N^{\frac{3\delta}2-\frac 12}}e^{C_3|\gamma^+-c^2N|},\text{ if }0<c<1 \text{ and }x<0,
\end{equation}
and
\begin{equation}
\label{eq:KEdgeLgC}
\left|K_{N,\gamma^+}(x,x)\right|\leq C_1 e^{-C_2 N^{\frac{3\delta}2-\frac 12}}e^{C_3|\gamma^+-c^2N|},\text{ if }1<c \text{ or }x>0.
\end{equation}
\end{lemma}
\begin{proof}
As before, we let $\tilde{\gamma}=\gamma^+-c^2N$ and drop the indices for $A$, $z^\pm$ and $\phi$ to simplify notation. The indices in this proof are $\frac x{cN}$. 

Suppose $0<c<1$ and $x<0$. Let $x=(c-2)cN-pcN^\delta,\ p>0$. It follows from \eqref{eq:defnA} that $A(z)$ has two distinct real critical points. Let $z^-$ be the smaller critical point. Similarly to \eqref{eq:asympOfZ+} we obtain
\begin{equation*}
z^-=\frac c{c-1}-\frac{c\sqrt{p}}{(1-c)^2}N^{\frac {\delta-1}2}+O(N^{\delta-1}).
\end{equation*}
If we deform the contours of integration of $K_{N,\gamma^+}(x,x)$ according to the saddle point method, one contour completely moves over the other. Thus, the residues we pick up total to $1$ and we have
\begin{equation*}
K_{N,\gamma^+}(x,x)-1=\frac{1}{(2\pi\mathfrak{i})^2}\oint_{C_u}\oint_{C_w}
\frac{e^{N(A(u)-A(z^-))}}{e^{N(A(w)-A(z^-))}}
\frac{e^{\tilde{\gamma}u^{-1}}}{e^{\tilde{\gamma}w^{-1}}}
\frac{1}{w}\frac{du\ dw}{u-w},
\end{equation*}
where the contours $C_u$ and $C_w$ are as in the left part of Figure \ref{fig:contDefEdgeSmC}.
\begin{figure}[t]
\includegraphics[width=13cm]{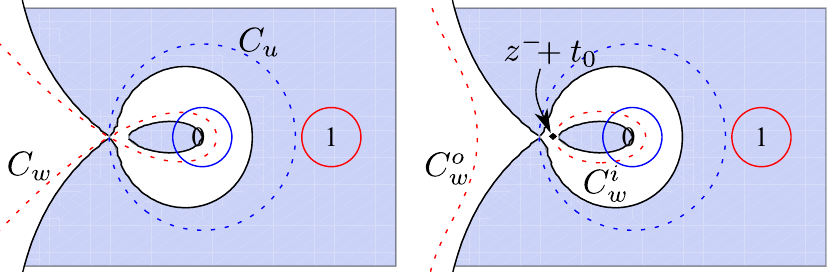}
\caption{\label{fig:contDefEdgeSmC} Deformation of contours near the left edge when $0<c<1$. $A(z)$ has two distinct real critical points. The shaded region corresponds to $\Re(A(z)-A(z^-))<0$. The solid red(right) and blue(left) contours are the original contours. The dotted red and blue contours are the deformed contours. The shaded region is bounded and the dotted red contour loops around it (not visible from the figures).}
\end{figure}
Without changing the integral, the contour $C_w$ can be further deformed into two closed contours $C_w^o$ and $C_w^i$ as in the right part of Figure \ref{fig:contDefEdgeSmC}. The outer contour $C_w^o$ can be moved so that there exists a constant $C_2>0$ such that $\Re(A(w)-A(z^-))>C_2$ for all $w$ along this contour. Since $u$ and $w$ are bounded away from $0$ and $1$, we obtain 
\begin{equation*}
\left|\frac{1}{(2\pi\mathfrak{i})^2}\oint_{C_u}\oint_{C_w^o}
\frac{e^{N(A(u)-A(z^-))}}{e^{N(A(w)-A(z^-))}}
\frac{e^{\tilde{\gamma}u^{-1}}}{e^{\tilde{\gamma}w^{-1}}}
\frac{1}{w}\frac{du\ dw}{u-w}\right|\leq C_1e^{-C_2N}e^{C_3|\tilde{\gamma}|},
\end{equation*}
for some constants $C_1,C_3>0$.

Since $z^-$ is a critical point of $A(z)$, it follows from Taylor's theorem that
\begin{equation*}
A(z^-+t)-A(z^-)=\frac 12 t^2 A''(z^-)+\frac 16 t^3 A'''(z^-)+O(t^4).
\end{equation*}
Similarly to \eqref{eq:asympOfA''} and \eqref{eq:asympOfA'''} we obtain
\begin{equation*}
A''(z^-)=-\frac{2(c-1)^3}{c}\sqrt{p}N^{\frac{\delta-1}2}+O(N^{\delta-1})>0
\end{equation*}
and
\begin{equation*}
A'''(z^-)=2\frac{(c-1)^5}{c^2}+\frac{6(c+2)(c-1)^4}{c^2}\sqrt{p}N^{\frac{\delta-1}2}+O(N^{\delta-1}),
\end{equation*}
which imply that there exist constants $D_1,D_2>0$, depending only on $c$ and $p$, such that for $t_0=D_1N^{\frac{\delta-1}2}$ we have 
\begin{equation*}
A(z^-+t_0)-A(z^-)=D_2t_0^2N^{\frac{\delta-1}2}+O(N^{\delta-1})
\end{equation*}
and 
\begin{equation*}
A(z^-+t)-A(z^-)>0\quad\text{ for all }t\in(0,t_0].
\end{equation*}
Thus, the inner contour $C_w^i$ can be chosen so that 
\begin{equation*}
\Re(A(w)-A(z^-))= D_2t_0^2N^{\frac{\delta-1}2}\quad\text{ for all }w\in C_w^i,
\end{equation*}
and $|u-w|\geq D_3 t_0$ for some constant $D_3$ and all $u\in C_u$, $w\in C_w^i$. Hence, there are constants $C_1,C_2,C'_2,C_3>0$ such that
\begin{multline*}
\left|\frac{1}{(2\pi\mathfrak{i})^2}\oint_{C_u}\oint_{C_w^i}
\frac{e^{N(A(u)-A(z^-))}}{e^{N(A(w)-A(z^-))}}
\frac{e^{\tilde{\gamma}u^{-1}}}{e^{\tilde{\gamma}w^{-1}}}
\frac{1}{w}\frac{du\ dw}{u-w}\right|
\\ \leq
\frac{C_1}{t_0}e^{-N\left(C_2t_0^2N^{\frac{\delta-1}2}\right)}e^{C_3|\tilde{\gamma}|}
\leq C_1e^{-C'_2N^{\frac{3\delta}2-\frac 12}}e^{C_3|\tilde{\gamma}|}.
\end{multline*}
This completes the proof of \eqref{eq:KEdgeSmC}.  The argument for \eqref{eq:KEdgeLgC} is similar.

$\qed$
\end{proof}

\begin{proposition}
\label{prop:PnNNearEdges}
Let $l(\lambda)$ denote the length of $\lambda$, i.e. the number of nonzero entries in $\lambda$, or equivalently the number of rows in its diagram. For any $\delta_0>\frac 13$ there exist constants $C_1,C_2>0$ such that for all $\delta\in[\delta_0,1)$, for all $N\in\mathbb{N}$ and for $n=\lfloor c^2N^2\rfloor$ we have
\begin{equation*}
\mathbb{P}_N^n(\{\lambda:l(\lambda)>(2-c)cN+N^{\delta}\}) \leq C_1 e^{-C_2 N^{\frac{3\delta}2-\frac 12}},\quad\text{ if }0<c<1,
\end{equation*}
\begin{equation*}
\mathbb{P}_N^n(\{\lambda:\lambda_N<N+(c-2)cN-N^{\delta}\}) \leq C_1 e^{-C_2 N^{\frac{3\delta}2-\frac 12}},\quad \text{ if }1<c,
\end{equation*}
and
\begin{equation*}
\mathbb{P}_N^n(\{\lambda:\lambda_1>(2+c)cN+N^{\delta}\}) \leq C_1 e^{-C_2 N^{\frac{3\delta}2-\frac 12}},\quad \text{ if }0<c.
\end{equation*}
\end{proposition}
\begin{proof}
Throughout the proof $C_1$ and $C_2$ denote arbitrary constants that depend only on $\delta_0$. Since $l(\lambda)>(2-c)cN+N^{\delta}$ implies that there exists $x\in[-N,(c-2)cN-N^{\delta}]$ such that $c_x(\lambda)=0$, we obtain
\begin{equation}
\label{eq:lengthToPoints}
\mathbb{P}_N^n(\{\lambda:l(\lambda)>(2-c)cN+N^{\delta}\}) \leq \sum_{x\in[-N,(c-2)cN-N^{\delta}]}\left(1-\mathbb{E}_{\mathbb{P}_N^n}(c_x)\right).
\end{equation}
When $0<c<1$ and ${x\in[-N,(c-2)cN-N^{\delta}]}$, by Lemma \ref{lem:KExtremes} we obtain
\begin{equation*}
|1-\mathbb{E}_{\mathbb{P}_N^{\gamma^+,0}}(c_x)|=|1-K_{N,\gamma^+}(x,x)|\leq C_1 e^{-C_2 N^{\frac{3\delta}2-\frac 12}}e^{C_3|\gamma^+-c^2N|}.
\end{equation*}
Depoissonizing by Lemma \ref{lem:depoissonization} we obtain
\begin{equation*}
|1-\mathbb{E}_{\mathbb{P}_N^n}(c_x)|\leq C_1 e^{-C_2 N^{\frac{3\delta}2-\frac 12}}.
\end{equation*}
This implies the first statement of the proposition since the index set in the sum in \eqref{eq:lengthToPoints} is of order $N$.

To prove the second statement, notice that
\begin{equation*}
\mathbb{P}_N^n(\{\lambda:\lambda_N<N+(c-2)cN-N^{\delta}\})\leq \sum_{x\in[-N,(c-2)cN-N^{\delta}]}\mathbb{E}_{\mathbb{P}_N^n}(c_x)
\end{equation*}
and proceed as above.

The last statement of the proposition can be proven in a similar way. 

$\qed$
\end{proof}
{\it Note:} The last statement of Proposition \ref{prop:PnNNearEdges} also follows immediately from Theorem 1.7 in \cite{JohDisc2001}, where it is proven that after appropriate scaling the local fluctuations of the longest row are characterized by the Tracy-Widom distribution.

Let $\mathcal{L}_\lambda(x)$ be the boundary of the rotated Young diagram $\lambda$ when it is scaled so that the cells have diagonal $2$. We have $\mathcal{L}_\lambda(x)=\sqrt{n}L_\lambda\left(\frac x{\sqrt{n}}\right)$. For $\delta>0$ and $K>0$ denote
\begin{equation*}
\mathbb{Y}^n_N(K,\delta)=
\left\{\lambda\in\mathbb{Y}^n_N:\supp |\mathcal{L}_\lambda(x)-|x||\subset \mathcal{I}^+_N(K,\delta) \right\}.
\end{equation*}
Figure \ref{fig:YnNKd} illustrates the restrictions put on the Young diagrams in the set $\mathbb{Y}^n_N(K,\delta)$.
\begin{figure}[t]
\includegraphics[width=14cm]{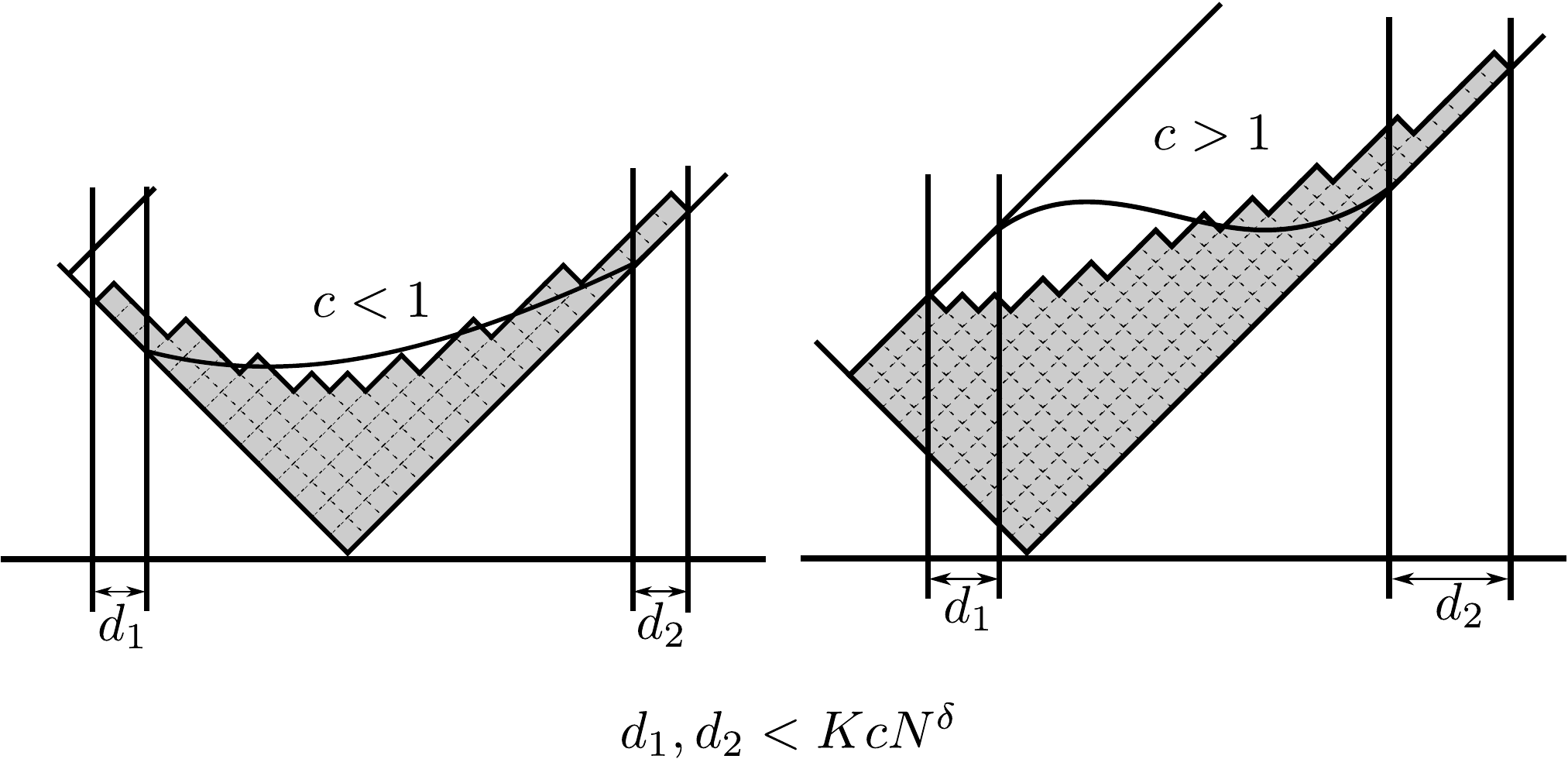}
\caption{\label{fig:YnNKd}Restrictions on the Young diagrams in the set $\mathbb{Y}^n_N(K,\delta)$. The curves represent the scaled limit shapes.}
\end{figure}
\begin{corollary}
\label{cor:PnNNearEdges}
For any $\delta_0>\frac 13$ there exist constants $C_1,C_2>0$ such that for all $\delta\in[\delta_0,1)$, for all $N\in\mathbb{N}$ and for $n=\lfloor c^2N^2\rfloor$ we have
\begin{equation*}
\mathbb{P}_N^n(\mathbb{Y}^n_N\backslash\mathbb{Y}^n_N(K,\delta)) \leq C_1 e^{-C_2 N^{\frac{3\delta}2-\frac 12}},\quad\text{ if }0<c<1,
\end{equation*}
\end{corollary}
\begin{proof}
This is essentially a reformulation of Proposition \ref{prop:PnNNearEdges}.

$\qed$
\end{proof}

\section{Estimates of the correlation kernel}
\label{sec:boundForCorrKer}
We need to estimate the decay of correlations. For this purpose a different representation of the correlation kernel is useful. In this section we obtain this representation and use it to obtain various estimates for the correlation kernel.

Define the functions
\begin{equation*}
K_{x,N}^{+}(\gamma^+)=\frac 1{2\pi\mathfrak{i}}\oint e^{\gamma^+ u^{-1}}(1-u)^N u^{x}\ du
\end{equation*}
where integration is over any closed counter--clockwise contour winding once around $0$, and
\begin{equation*}
K_{y,N}^{-}(\gamma^+)=\frac 1{2\pi\mathfrak{i}}\oint e^{-\gamma^+ w^{-1}}(1-w)^{-N} w^{-y}\ dw
\end{equation*}
where integration is over any closed counter--clockwise contour winding once around $1$ and not containing $0$.
\begin{lemma}
\label{lem:kerKbyFuncK}
If $x\neq y$, then
\begin{equation}
\label{eq:KDiffForm}
K_{N,\gamma^+}(x,y)=\frac{N K_{x,N-1}^{+}(\gamma^+)K_{y+1,N+1}^{-}(\gamma^+)-\gamma^+ K_{x-1,N}^{+}(\gamma^+)K_{y+2,N}^{-}(\gamma^+)}{x-y}.
\end{equation}
\end{lemma}
\begin{proof}
The main idea of the proof is to integrate formula \eqref{eq:intFormK} by parts (the idea was used by A. Okounkov to obtain a similar formula for the Bessel kernel \cite{OkNASA2001}).

In general, for functions $f(u,w)$ and $g(u,w)$ which are differentiable on simple closed contours $C_u$ and $C_w$ integration by parts gives
\begin{align*}
\oint_{C_u}\oint_{C_w}&f\left(u\frac\partial{\partial u}+w\frac\partial{\partial w}\right)g\ dwdu 
\\&= \oint_{C_w}\oint_{C_u}f u\frac\partial{\partial u}g\ dudw
+ \oint_{C_u}\oint_{C_w}fw\frac\partial{\partial w} g\ dwdu
\\&= - \oint_{C_w}\oint_{C_u}g\left(u\frac\partial{\partial u}+1\right)f\ dudw
- \oint_{C_u}\oint_{C_w}g\left(w\frac\partial{\partial w}+1\right) f\ dwdu
\\&= - \oint_{C_u}\oint_{C_w}g\left(u\frac\partial{\partial u}+w\frac\partial{\partial w}+2\right) f\ dwdu.
\end{align*}

Since
\begin{equation*}
u^xw^{-y-1}=\left(u\frac \partial{\partial u} + w\frac \partial{\partial w}+1\right)\frac{u^x w^{-y-1}}{x-y},
\end{equation*}
applying the integration by parts calculation above to \eqref{eq:intFormK} we obtain
% \begin{equation*}
% K_{N,\gamma^+}(x,y)=-\frac{1}{(2\pi\mathfrak{i})^2}\oint_{C_u}\oint_{C_w}\frac{u^x w^{-y-1}}{x-y}
% \left(u\frac \partial{\partial u} + w\frac \partial{\partial w}+1\right)
% \frac{e^{\gamma^+(u^{-1}-w^{-1})}\frac{(1-u)^N}{(1-w)^{N}}}{u-w}\ du\ dw.
% \end{equation*}
\begin{multline*}
K_{N,\gamma^+}(x,y)=-\frac{1}{(2\pi\mathfrak{i})^2}\oint_{C_u}\oint_{C_w}\frac{u^x w^{-y-1}}{x-y}
\\\times\left(u\frac \partial{\partial u} + w\frac \partial{\partial w}+1\right)
\frac{e^{\gamma^+(u^{-1}-w^{-1})}\frac{(1-u)^N}{(1-w)^{N}}}{u-w}\ du\ dw.
\end{multline*}
It follows from
% \begin{equation*}
% \left(u\frac \partial{\partial u} + w\frac \partial{\partial w}+1\right)\frac{e^{\gamma^+(u^{-1}-w^{-1})}\frac{(1-u)^N}{(1-w)^{N}}}{u-w}
% =e^{\gamma^+(u^{-1}-w^{-1})}\frac{(1-u)^N}{(1-w)^{-N}}\left(\frac{\gamma^+}{uw}-\frac{N}{(1-u)(1-w)}\right)
% \end{equation*}
\begin{multline*}
\left(u\frac \partial{\partial u} + w\frac \partial{\partial w}+1\right)\frac{e^{\gamma^+(u^{-1}-w^{-1})}\frac{(1-u)^N}{(1-w)^{N}}}{u-w}
\\=e^{\gamma^+(u^{-1}-w^{-1})}\frac{(1-u)^N}{(1-w)^{-N}}\left(\frac{\gamma^+}{uw}-\frac{N}{(1-u)(1-w)}\right)
\end{multline*}
that the integrals with respect to $u$ and $w$ can be separated. Carrying this out we obtain \eqref{eq:KDiffForm}.
$\qed$
\end{proof}

\subsection{Estimates of \texorpdfstring{$K_{x,N}^{\pm}(\gamma^+)$}{K\_\{x,N,+,-\}} for various values of \texorpdfstring{$x$}{x}}

\begin{lemma}
\label{lem:estFuncK+}
%Suppose $c\neq 1$.
For any $\delta_0>\frac 13$ there exist constants $C_1=C_1(\delta_0)>0$ and $C_2=C_2(\delta_0)>0$ such that 
\begin{equation*}
\left|K_{x,N}^{\pm}(\gamma^+)\right|\leq C_1 e^{\pm N \Re A_{\frac x{cN}}\left(z_{\frac x{cN}}^+\right)} \frac{e^{C_2|\gamma^+-c^2N|}}{N^{\frac {\delta+1}4}}
\end{equation*}
for all $\delta\in[\delta_0,1)$, all $x\in\mathcal{I}^-_N(1,\delta)$, all $\gamma^+$ and all $N\in\mathbb{N}$.
\end{lemma}
%{\it Note:} The condition $c\neq 1$ is only needed near the left edge. If $-(2-\varepsilon)\leq\frac x{cN}-c\leq 2-N^{\delta-1}$, for $\varepsilon>0$ independent of $N$, then the result holds for $c=1$ as well.

\begin{proof}
We present the proof of the result for $K_{x,N}^{+}(\gamma^+)$. The proof for $K_{x,N}^{-}(\gamma^+)$ is completely identical.

Throughout the proof $C_1$ and $C_2$ will denote arbitrary constants that depend only on $\delta_0$. We will use the same notation as in the proof of Lemma \ref{lem:depoisK2}. In particular $\tilde{\gamma}=\gamma^+-c^2N$, $A(u)=A_{\frac x{cN}}(u)$, $z^\pm=z_{\frac x{cN}}^\pm$, the contour of integration is deformed so that it goes through $z^\pm$ and has the property that for all $u$ on the deformed contour $\Re(A(u)-A(z^+))\leq 0$ and the deformed contour $C_u$ is divided into three parts as in \eqref{eq:contourPieces}. Consider
\begin{equation*}
K_{x,N}^{+}(\gamma^+) e^{-N A(z^+)}=\frac 1{2\pi\mathfrak{i}}\oint_{C_u}e^{N(A(u)-A(z^+))}e^{\tilde{\gamma}u^{-1}}\ du.
\end{equation*}
Let $\frac x{cN}-c=\pm(2-pN^{\delta-1})$ for some $p>0$. Arguments similar to those in the proof of Lemma \ref{lem:depoisK2} show that the contribution of the large contour $C'_u$ is exponentially small. Let $\beta$ be as in Lemma \ref{lem:depoisK2}. On the contour $C_{u,+}$ we have
\begin{align*}
\left|\frac 1{2\pi\mathfrak{i}}\oint_{C_{u,+}}e^{N(A(u)-A(z^+))}e^{\tilde{\gamma}u^{-1}} \ du\right|
&\leq C_1 e^{C_2|\tilde{\gamma}|}\int_{-N^{-\beta}}^{N^{-\beta}}e^{N\Re(A(z^++e^{\mathfrak{i}\xi}t)-A(z^+))}\ dt
\\&\leq C_1 e^{C_2|\tilde{\gamma}|}\int_{-N^{-\beta}}^{N^{-\beta}}e^{-D\sqrt{p}N^{\frac{\delta+1}2}t^2}\ dt
%=C_1 e^{C_2|\tilde{\gamma}|}\int_{-N^{-\beta}}^{N^{-\beta}}e^{-D Nt^2}O(1)dt
\end{align*}
for some positive constant $D$.
Making the change of variable $t'=N^{\frac {\delta+1}4}t$ we obtain
\begin{equation*}
\left|\frac 1{2\pi\mathfrak{i}}\oint_{C_{u,+}}e^{N(A(u)-A(z^+))}e^{\tilde{\gamma}u^{-1}} \ du\right|
\leq C_1 \frac{e^{C_2|\tilde{\gamma}|}}{N^{\frac {\delta+1}4}}\int_{-N^{\frac {\delta+1}4-\beta}}^{N^{\frac {\delta+1}4-\beta}}e^{-D\sqrt{p}t'^2}\ dt'\leq C_1 \frac{e^{C_2|\tilde{\gamma}|}}{N^{\frac {\delta+1}4}}.
\end{equation*}
Of course, the contribution from $C_{u,-}$ is of the same order.

$\qed$
\end{proof}

\begin{lemma}
\label{lem:estDiffK+}
%Suppose $c\neq 1$.
For any $\delta_0>\frac 13$ there exist constants $C_1=C_1(\delta_0)>0$ and $C_2=C_2(\delta_0)>0$ such that 
\begin{equation*}
\left|K_{x+1,N}^{\pm}(\gamma^+)-\sign(x) cK_{x,N+1}^{\pm}(\gamma^+)\right|\leq C_1 e^{\pm N \Re A_{\frac x{cN}}\left(z_{\frac x{cN}}^+\right)} \frac{e^{C_2|\gamma^+-c^2N|}}{N^{\frac {3-\delta}4}}
\end{equation*}
for all $\delta\in[\delta_0,1)$, all $x\in\mathcal{I}^-_N(1,\delta)$, all $\gamma^+$ and all $N\in\mathbb{N}$.
\end{lemma}
%{\it Note:} The condition $c\neq 1$ is only needed near the left edge. 

\begin{proof}
We present the proof of the result for $K_{x,N}^{+}(\gamma^+)$. The proof for $K_{x,N}^{-}(\gamma^+)$ is completely identical.

In this proof the indices of $A(u)$ and $z^+$ are $\frac x{cN}$. The proof is similar to the proof of Lemma \ref{lem:estFuncK+}. Suppose $\frac x{cN}-c=(2-pN^{\delta-1})>0$. 
We have 
% \begin{equation*}
% \left(K_{x+1,N}^{+}(\gamma^+)-cK_{x,N+1}^{+}(\gamma^+)\right)e^{-N A(z^+)}
% =\frac 1{2\pi\mathfrak{i}}\oint_{C}e^{N(A(u)-A(z^+))}e^{\tilde{\gamma}u^{-1}}(u-c(1-u)) \ du.
% \end{equation*}
\begin{multline*}
\left(K_{x+1,N}^{+}(\gamma^+)-cK_{x,N+1}^{+}(\gamma^+)\right)e^{-N A(z^+)}
\\=\frac 1{2\pi\mathfrak{i}}\oint_{C}e^{N(A(u)-A(z^+))}e^{\tilde{\gamma}u^{-1}}(u-c(1-u)) \ du.
\end{multline*}
The main contribution comes from the sections of the contour near $z^{\pm}$. If $u=z^++e^{\mathfrak{i}\xi}t$, then from \eqref{eq:asympOfZ+} we obtain 
\begin{equation*}
|u-c(1-u)|=(c+1)\left|z^++e^{\mathfrak{i}\xi}t-\frac c{c+1}\right|\leq D_1N^{\frac{\delta-1}2}+D_2 t
\end{equation*}
for some positive constants $D_1$, $D_2$. Proceeding as in Lemma \ref{lem:estFuncK+}, we obtain
% \begin{equation*}
% \left|K_{x+1,N}^{+}(\gamma^+)-cK_{x,N+1}^{+}(\gamma^+)\right|
% \leq C_1 e^{N \Re A_{\frac x{cN}}\left(z_{\frac x{cN}}^+\right)} \frac{e^{C_2|\gamma^+-c^2N|}}{N^{\frac {\delta+1}4}}\left(D_1N^{\frac{\delta-1}2}+\frac{D_2}{N^{\frac{\delta+1}4}}\right),
% \end{equation*}
\begin{multline*}
\left|K_{x+1,N}^{+}(\gamma^+)-cK_{x,N+1}^{+}(\gamma^+)\right|
\\\leq C_1 e^{N \Re A_{\frac x{cN}}\left(z_{\frac x{cN}}^+\right)} \frac{e^{C_2|\gamma^+-c^2N|}}{N^{\frac {\delta+1}4}}\left(D_1N^{\frac{\delta-1}2}+\frac{D_2}{N^{\frac{\delta+1}4}}\right),
\end{multline*}
which completes the proof when $x>0$.

When $x<0$, instead of $u-c(1-u)$ we have
\begin{equation*}
|u+c(1-u)|=|c-1|\left|z^++e^{\mathfrak{i}\xi}t-\frac c{c-1}\right|,
\end{equation*}
and the rest follows as above, since in this case it follows from \eqref{eq:asympOfZ+} that the leading term of $z^+$ is $\frac c{c-1}$.

$\qed$
\end{proof}

\begin{lemma}
\label{lem:estFuncK+OnEdge}
%Suppose $c\neq 1$.
There exist constants $C_1>0$ and $C_2>0$ such that
\begin{equation*}
\left|K_{x,N}^{\pm}(\gamma^+)\right|\leq C_1 e^{\pm N \Re A_{\frac x{cN}}\left(z_{\frac x{cN}}^+\right)} \frac{e^{C_2|\gamma^+-c^2N|}}{N^{\frac 13}}
\end{equation*}
for all $x$, all $\gamma^+$ and all $N\in\mathbb{N}$.
\end{lemma}
\begin{proof}
We present the proof of the result for $K_{x,N}^{+}(\gamma^+)$. %The proof for $K_{x,N}^{-}(\gamma^+)$ is completely identical.
As before, we drop the indices of $A(z)$ and $z^+$, which are $\frac x{cN}$ in this proof, and let $\tilde{\gamma}=\gamma^+-c^2N$. Let $\left|\frac x{cN}-c\right|=2- pN^{\delta-1}$, $\delta\geq 0$. 

Suppose $p>0$. In this case $A(z)$ has complex conjugate critical points. We deform the integration contour as contour $C_u$ in Lemma \ref{lem:depoisK2}, however with one difference: near the critical points we deform the contour to be piecewise linear with {\it different} slopes on each side of the critical points. More precisely, let $\xi_{1,2}\in(0,\pi)$ and deform the integration contour so that it is given by $z^\pm+e^{\pm\mathfrak{i}\xi_1}t,t>0$ and $z^\pm-e^{\pm\mathfrak{i}\xi_2}t, t>0$ near the critical points $z^\pm$. Choose $\xi_{1,2}$ so that both $\Re(e^{2\mathfrak{i}\xi}A''(z^+))<0$ and $\Re(e^{3\mathfrak{i}\xi}A'''(z^+))<0$. For example, when $\frac x{cN}-c=2- pN^{\delta-1}$, it follows from \eqref{eq:asympOfA''} and \eqref{eq:asympOfA'''} that $\frac \pi 2<\xi_{1,2}<\frac{5\pi}6$. 
Consider
\begin{equation*}
K_{x,N}^{+}(\gamma^+) e^{-N A(z^+)}=\frac 1{2\pi\mathfrak{i}}\oint_{C}e^{N(A(u)-A(z^+))}e^{\tilde{\gamma}u^{-1}}\ du.
\end{equation*}
We divide the contour into five sections: one away from the critical points and two linear sections near each critical point. That the contribution of the contour away from the critical points is exponentially small, can be seen as in Lemma \ref{lem:depoisK2}. The contribution of the linear sections near the critical points is of order
\begin{equation*}
B(N)=\int_0^\varepsilon e^{-N(D_1t^2N^{\frac{\delta-1}2}+D_2t^3)}e^{\tilde{\gamma}\Re\left((z^\pm+e^{\pm\mathfrak{i}\xi_{1,2}}t)^{-1}\right)}\ dt
\end{equation*}
for some positive constants $D_1$ and $D_2$. We estimate $B(N)$ as follows:
\begin{equation*}
B(N)\leq e^{C_2\tilde{\gamma}} \int_0^\varepsilon e^{-D_2Nt^3}\ dt=\frac{e^{C_2\tilde{\gamma}}}{N^{\frac 13}} \int_0^{\varepsilon N^{\frac 13}} e^{-D_2s^3}\ ds\leq C_1\frac{e^{C_2\tilde{\gamma}}}{N^{\frac 13}}.
\end{equation*}

When $p<0$, $A(z)$ has two real critical points. We deform the integration contour as in Lemma \ref{lem:KExtremes} and proceed as above.

$\qed$
\end{proof}

\begin{lemma}
\label{lem:estFuncK+Extremes}
%Suppose $c\neq 1$.
For any $\delta>\frac 13$ there exist constants $C_1>0$, $C_2>0$ and $C_3>0$ such that
\begin{equation*}
\left|K_{x,N}^{\pm}(\gamma^+)\right|\leq C_1 e^{\pm N \Re A_{\frac x{cN}}\left(z_{\frac x{cN}}^\pm\right)} e^{-C_3N^{\frac {3\delta}2-\frac 12}}e^{C_2|\gamma^+-c^2N|}
\end{equation*}
for all $x\notin\mathcal{I}^+_N(1,\delta)$, $x>-N$, all $\gamma^+$ and all $N\in\mathbb{N}$.
\end{lemma}
\begin{proof}
For $K_{x,N}^{-}(\gamma^+)$ deform the integration contour as contour $C_w$ in Lemma \ref{lem:KExtremes} and estimate the contour integral as in Lemma \ref{lem:KExtremes}. The only difference in obtaining the estimate for $K_{x,N}^{+}(\gamma^+)$ is that the contour should be deformed to pass through the {\it larger} of the two real critical points of $A_{\frac x{cN}}(z)$.

$\qed$
\end{proof}

\subsection{Several estimates of the correlation kernel}
In this section we use the estimates of the functions $K^{\pm}_{x,N}$ obtained in the previous section to obtain estimates for the correlation kernel.

\begin{lemma}
\label{lem:decayCorrK}
%Suppose $c\neq 1$.
For any $\varepsilon>0$ there exist constants $C_1$ and $C_2$ such that 
\begin{equation*}
\left|K_{N,\gamma^+}(x,y)\right|\leq C_1\frac{e^{C_2|\gamma^+-c^2N|}}{1+|x-y|}e^{N\Re\left(A_{\frac x{cN}}\left(z_{\frac x{cN}}^+\right)-A_{\frac y{cN}}\left(z_{\frac y{cN}}^+\right)\right)}
\end{equation*}
for all $x,y\in\mathcal{I}_N(\varepsilon)$, all $\gamma^+$ and all $N\in\mathbb{N}$.
\end{lemma}
\begin{proof}
When $x\neq y$ this follows from Lemmas \ref{lem:kerKbyFuncK} and \ref{lem:estFuncK+}. When $x=y$ the result follows from Lemma \ref{lem:depoisK2}.

$\qed$
\end{proof}

\begin{lemma}
\label{lem:estKDiffDelta}
Let $K_1$ and $K_2$ be arbitrary positive constants, let $\frac 13<\delta_0\leq \delta_1,\delta_2\leq 1$, and let $x\in\mathcal{I}^-_N(K_1,\delta_1)$, $y\in\mathcal{I}^-_N(K_2,\delta_2)$. There exist constants $C_1,C_2>0$, which depend only on $K_1,K_2$ and $\delta_0$, such that for all $\gamma^+$ and for all $N\in\mathbb{N}$ the following hold. If $x$ and $y$ have the same sign, then
\begin{equation}
\label{eq:KsqEdgeSameSign}
|K_{N,\gamma^+}(x,y)K_{N,\gamma^+}(y,x)|\leq C_1e^{C_2|\gamma^+-c^2N|}\frac{N^{\frac{|\delta_1-\delta_2|}2}}{(1+|x-y|)^2}.
\end{equation}
If $x$ and $y$ have opposite signs, then
\begin{equation}
\label{eq:KsqEdgeOpSign}
|K_{N,\gamma^+}(x,y)K_{N,\gamma^+}(y,x)|\leq C_1e^{C_2|\gamma^+-c^2N|}\frac{N^{1-\frac{\delta_1+\delta_2}2}}{(1+|x-y|)^2}.
\end{equation}
\end{lemma}
\begin{proof}
If $x=y$, the result follows immediately from Lemma \ref{lem:depoisK2}. If $x\neq y$ and they have the same sign, from Lemma \ref{lem:kerKbyFuncK} we obtain
\begin{align*}
K_{N,\gamma^+}(x,y)
=&\frac{1}{x-y}\bigg(N K_{x,N-1}^{+}(\gamma^+)\Big(K_{y+1,N+1}^{-}(\gamma^+)-\sign(y)cK_{y+2,N}^{-}(\gamma^+)\Big)
\\&+\sign(y)Nc\Big(K_{x,N-1}^{+}(\gamma^+)-\sign(x)cK_{x-1,N}^{+}(\gamma^+)\Big)K_{y+2,N}^{-}(\gamma^+)
\\&+(\gamma^+-c^2N) K_{x-1,N}^{+}(\gamma^+)K_{y+2,N}^{-}(\gamma^+)\bigg).
\end{align*}
If $|\gamma^+-c^2N|<N^{\frac 12}$, applying Lemmas \ref{lem:estFuncK+} and \ref{lem:estDiffK+} we obtain \eqref{eq:KsqEdgeSameSign}. If $|\gamma^+-c^2N|>N^{\frac 12}$, applying the same lemmas we obtain
\begin{equation*}
|K_{N,\gamma^+}(x,y)K_{N,\gamma^+}(y,x)|\leq C_1e^{C_2|\gamma^+-c^2N|}\frac{|\gamma^+-c^2N|^2N^{-\frac{\delta_1+\delta_2+2}2}}{(1+|x-y|)^2},
\end{equation*}
which implies \eqref{eq:KsqEdgeSameSign} with a larger $C_2$.

When $x$ and $y$ have opposite signs, \eqref{eq:KsqEdgeOpSign} follows from Lemmas \ref{lem:kerKbyFuncK} and \ref{lem:estFuncK+}.

$\qed$
\end{proof}
\begin{remark}
\label{rem:KsqEdge2}
Notice that one of the sets $\mathcal{I}^-_N(K_1,\delta_1)$, $\mathcal{I}^-_N(K_2,\delta_2)$ is contained in the other. If, for example, $\mathcal{I}^-_N(K_1,\delta_1)\subset\mathcal{I}^-_N(K_2,\delta_2)$, then both $x$ and $y$ are in $\mathcal{I}^-_N(K_2,\delta_2)$, whence \eqref{eq:KsqEdgeSameSign} implies the better estimate
\begin{equation*}
%\label{eq:KsqEdge2}
|K_{N,\gamma^+}(x,y)K_{N,\gamma^+}(y,x)|\leq \frac{C_1e^{C_2|\gamma^+-c^2N|}}{(1+|x-y|)^2}.
\end{equation*}
\end{remark}

\begin{lemma}
\label{lem:estKBulkEdge}
Let $K>0$ and $\frac 13<\delta_0\leq \delta_1<1$. There exist constants $C_1,C_2>0$, which depend only on $K$ and $\delta_0$, such that for all $\gamma^+$, for all $x\in\mathcal{I}^-_N(K_1,\delta_1)$, for all $y$ and for all $N\in\mathbb{N}$ we have
\begin{equation*}
|K_{N,\gamma^+}(x,y)K_{N,\gamma^+}(y,x)|\leq C_1e^{C_2|\gamma^+-c^2N|}\frac{N^{\frac{5-3\delta_1}{6}}}{(1+|x-y|)^2}.
\end{equation*}
\end{lemma}
\begin{proof}
The lemma follows immediately from Lemmas \ref{lem:kerKbyFuncK}, \ref{lem:estFuncK+} and \ref{lem:estFuncK+OnEdge}.

$\qed$
\end{proof}

\begin{lemma}
\label{lem:estKExtremes}
Let $K_1$ and $K_2$ be arbitrary positive constants, let $\frac 13<\delta_0\leq \delta_1<1$, $\frac 13\leq \delta_2$, and let $x\in\mathcal{I}^-_N(K_1,\delta_1)$, $y\notin\mathcal{I}^+_N(K_2,\delta_2)$, $y>-N$. There exist constants $C_1,C_2,C_3>0$, which depend only on $K_1,K_2$ and $\delta_0$, such that for all $\gamma^+$ and for all $N\in\mathbb{N}$ we have
\begin{equation*}
|K_{N,\gamma^+}(x,y)K_{N,\gamma^+}(y,x)|\leq C_1e^{C_2|\gamma^+-c^2N|}\frac{e^{-C_3N^{\frac {3\delta}2-\frac 12}}}{(1+|x-y|)^2}.
\end{equation*}
\end{lemma}
\begin{proof}
Lemmas \ref{lem:kerKbyFuncK}, \ref{lem:estFuncK+} and \ref{lem:estFuncK+Extremes} imply
\begin{equation*}
\left|K_{N,\gamma^+}(x,y)\right|\leq C_1\frac{e^{C_2|\gamma^+-c^2N|}e^{-C_3N^{\frac {3\delta}2-\frac 12}}}{1+|x-y|}e^{N\Re\left(A_{\frac x{cN}}\left(z_{\frac x{cN}}^+\right)-A_{\frac y{cN}}\left(z_{\frac y{cN}}^+\right)\right)},
\end{equation*}
while Lemmas \ref{lem:kerKbyFuncK}, \ref{lem:estFuncK+} and \ref{lem:estFuncK+OnEdge} imply
\begin{equation*}
\left|K_{N,\gamma^+}(y,x)\right|\leq C_1\frac{e^{C_2|\gamma^+-c^2N|}N^{\frac{5-3\delta_1}{12}}}{1+|x-y|}e^{N\Re\left(A_{\frac y{cN}}\left(z_{\frac y{cN}}^+\right)-A_{\frac x{cN}}\left(z_{\frac x{cN}}^+\right)\right)}.
\end{equation*}
Combining the two estimates completes the proof.

$\qed$
\end{proof}

\subsection{Decay of correlations in the bulk}
\label{subsec:decayOfCorr}
In this section we use the estimates of the correlation kernel obtained in the previous section to estimate the decay of correlations in the bulk.

\begin{proposition}
\label{prop:covK}
For any $\varepsilon>0$ and any integer $L>0$ there exist positive constants $C_1=C_1(\varepsilon,L)$ and $C_2=C_2(\varepsilon,L)$ such that 
\begin{align*}
Cov_{\mathbb{P}^{\gamma^+,0}_N}(x,\vec{l};y,\vec{m})
:&=\left|\mathbb{E}_{\mathbb{P}^{\gamma^+,0}_N}(c_{x+\vec{l}}\cdot c_{y+\vec{m}})-\mathbb{E}_{\mathbb{P}^{\gamma^+,0}_N}(c_{x+\vec{l}})\mathbb{E}_{\mathbb{P}^{\gamma^+,0}_N}(c_{y+\vec{m}})\right|
\\&\leq C_1\frac{e^{C_2|\gamma^+-c^2N|}}{(1+|x-y|)^2}
\end{align*}
for all $x,y\in\mathcal{I}_N(\varepsilon)$, all integer vectors $\vec{l}$ and $\vec{m}$ satisfying $|\vec{l}|,|\vec{m}|\leq L$, all $\gamma^+$ and all $N\in\mathbb{N}$.
\end{proposition}
\begin{proof}
It follows from Theorem \ref{thm:intFormK} that $\mathbb{E}_{\mathbb{P}^{\gamma^+,0}_N}(c_{x+\vec{l}}\cdot c_{y+\vec{m}})$ is a determinant of the form
\begin{equation*}
\mathbb{E}_{\mathbb{P}^{\gamma^+,0}_N}(c_{x+\vec{l}}\cdot c_{y+\vec{m}})=\det A
=\det\left(\begin{array}{cc}B&C\\ D&E\end{array}\right),
\end{equation*}
where $\mathbb{E}_{\mathbb{P}^{\gamma^+,0}_N}(c_{x+\vec{l}})=\det B$ and $\mathbb{E}_{\mathbb{P}^{\gamma^+,0}_N}(c_{y+\vec{m}})=\det E$. 

Thus, it follows that $Cov_{\mathbb{P}^{\gamma^+,0}_N}(x,\vec{l};y,\vec{m})$ consists of terms in $\det A$ which have at least one factor from each of $C$ and $D$. Since the terms in $C$ and $D$ are of the form $K_{N,\gamma^+}(x+l_i,y+m_j)$, the proposition follows from Lemma \ref{lem:decayCorrK}. Note that the factors 
\begin{equation*}
e^{N\Re\left(A_{\frac x{cN}}\left(z_{\frac x{cN}}^+\right)-A_{\frac y{cN}}\left(z_{\frac y{cN}}^+\right)\right)}
\end{equation*}
cancel out, since we are taking a determinant.
% \begin{equation*}
% \left(
% \begin{array}{cc}
% \framebox(70,70){$\mathbb{E}_{\mathbb{P}^{\gamma^+,0}_N}(c_{x+\vec{l}})$}&\framebox(70,70){*}
% \\ \framebox(70,70){*} &\framebox(70,70){$\mathbb{E}_{\mathbb{P}^{\gamma^+,0}_N}(c_{x+\vec{l}})$}
% \end{array}
% \right)
% \end{equation*}

$\qed$
\end{proof}
\begin{proposition}
\label{prop:covOfPnN}
For any $\varepsilon>0$ and any integer $L>0$ there exists a positive constant $C=C(\varepsilon,L)$ such that 
\begin{equation*}
\left|\mathbb{E}_{\mathbb{P}_N^n}(c_{x+\vec{l}}\cdot c_{y+\vec{m}})-\mathbb{E}_{\mathbb{P}_N^n}(c_{x+\vec{l}})\mathbb{E}_{\mathbb{P}_N^n}(c_{y+\vec{m}})\right|\leq \frac{C}{\min\{N,(1+|x-y|)^2\}}
\end{equation*}
for all $x,y\in\mathcal{I}_N(\varepsilon)$, all integer vectors $\vec{l}$ and $\vec{m}$ satisfying $|\vec{l}|,|\vec{m}|\leq L$, all $N\in\mathbb{N}$ and $n= \lfloor c^2N^2\rfloor$.
\end{proposition}
\begin{proof}
If $x\in\mathcal{I}_N(\varepsilon)$, then $2cN-|x-c^2N|$ is of order $N$. Using Lemma \ref{lem:depoisK2} and Proposition \ref{prop:covK}, and noting that the terms $\left(\frac{c}{\sqrt{1+\frac xN}}\right)^{-l}$ in Lemma \ref{lem:depoisK2} cancel since we are taking determinants, we obtain
\begin{equation*}
\left|\mathbb{E}_{\mathbb{P}^{\gamma^+,0}_N}(c_{x+\vec{l}}\cdot c_{y+\vec{m}})-\mathbb{E}_{\mathbb{S}(\phi_{\frac x{cN}})}(c_{x+\vec{l}})\mathbb{E}_{\mathbb{S}(\phi_{\frac y{cN}})}(c_{y+\vec{m}})\right|\leq \frac{C_1e^{C_2|\gamma^+-c^2N|}}{\min\{N,(1+|x-y|)^2\}}.
\end{equation*}
Depoissonizing by Lemma \ref{lem:depoissonization} we obtain
\begin{equation*}
\left|\mathbb{E}_{\mathbb{P}_N^n}(c_{x+\vec{l}}\cdot c_{y+\vec{m}})-\mathbb{E}_{\mathbb{S}(\phi_{\frac x{cN}})}(c_{x+\vec{l}})\mathbb{E}_{\mathbb{S}(\phi_{\frac y{cN}})}(c_{y+\vec{m}})\right|\leq \frac{C}{\min\{N,(1+|x-y|)^2\}}.
\end{equation*}
Applying Proposition \ref{prop:localStatOfPnN} to this expression completes the proof.

$\qed$
\end{proof}

\section{Proof of Theorem \ref{thm:main}}
\label{sec:mainProof}
In this section we present the proof of the main theorem. We evaluate the limit of the terms in \eqref{eq:intFormMeas} separately.
\begin{lemma}
\label{lem:limitForContents}
For any $\varepsilon>0$ we have
\begin{equation*}
\lim_{\substack{N\rightarrow \infty\\n=\lfloor c^2N^2\rfloor}}\mathbb{P}_N^n\left\{\lambda:\left|\hat{\rho}(\lambda)\right|<\varepsilon\right\}=1,
\end{equation*}
where $\hat{\rho}(\lambda)$ is as in Proposition \ref{prop:intFormMeas}.
\end{lemma}
\begin{proof}
Let $\mathfrak{c}_k(\lambda)$ be the number of cells in $\lambda$ with content $k$. Notice that if $\lambda\in\mathbb{Y}^n_N$, then $\mathfrak{c}_{k-N}(\lambda)\leq\min\{k,N\}$. Hence,
\begin{equation*}
\hat{\rho}(\lambda)=\sum_{k=1}^{\infty}\frac{\mathfrak{c}_{k-N}(\lambda)}{2\sqrt{n}}\mathfrak{m}(k)\leq \frac{1}{2\sqrt{n}} \sum_{k=1}^{N}\mathfrak{m}(k)k+\frac{N}{2\sqrt{n}} \sum_{k=N+1}^{\infty}\mathfrak{m}(k).
\end{equation*}
Differentiating $\mathfrak{m}(x)$ three times, we obtain
\begin{equation*}
\mathfrak{m}'''(x)=\sum_{k=1}^\infty \frac{4}{x^{2k+3}} = \frac{4}{x^3(x^2-1)},
\end{equation*}
whence there exists a constant $c>0$ such that
\begin{equation*}
\hat{\rho}(\lambda)\leq \frac{c \ln N}{2\sqrt n}+\frac c{2\sqrt n}.
\end{equation*}

$\qed$
\end{proof}

\begin{lemma}
\label{lem:sumsOfCs}
For any continuous bounded function $f:\mathbb{R}\rightarrow\mathbb{C}$, any integer vector $\vec{m}$, and any $\varepsilon>0$, we have the following convergence in measure:
\begin{multline*}
\lim_{\substack{N\rightarrow \infty\\n=\lfloor c^2N^2\rfloor}}
\mathbb{P}_N^n\left\{\lambda:\left|\frac 1{cN} \sum_{k=-N}^{\infty}f\left(\frac k{cN}\right)c_{k+\vec{m}}(\lambda)
\right.\right.\\\left.\left.-\left(\int_{c-2}^{c+2}f(a)\mathbb{E}_{\mathbb{S}(\phi_a)}c_{\vec{m}}\ da+\delta_{c<1}\int_{-\frac 1c}^{c-2}f(a)\ da\right)\right|<\varepsilon\right\}=1.
\end{multline*}
\end{lemma}
%{\it Note}: Limits of such weighted sum of the indicator function of appearance of a local pattern on the boundary of a Young diagram with respect to the Plancherel measure were obtained in \cite{Bu}.

\begin{proof}
Let $\varepsilon_0>0$ and $1>\delta>\frac 13$ be fixed. Throughout the proof, $C$ will denote an arbitrary constant that depends only on $\varepsilon_0$ and $f$.
% Define 
% \begin{equation*}
% A_{N,f,\vec{m}}(k)=f\left(\frac k{cN}\right)c_{k+\vec{m}}(\lambda)-f\left(\frac k{cN}\right)\mathbb{E}_{\mathbb{S}(\phi_{\frac k{cN}})}c_{\vec{m}}.
% \end{equation*}
It follows from Propositions \ref{prop:covOfPnN} and \ref{prop:localStatOfPnN} that 
\begin{multline*}
\left|f\left(\frac k{cN}\right)f\left(\frac l{cN}\right)
\mathbb{E}_{\mathbb{P}_N^n}\left(\left(c_{k+\vec{m}}-\mathbb{E}_{\mathbb{S}(\phi_{\frac k{cN}})}c_{\vec{m}}\right)\cdot\left(c_{k+\vec{m}}-\mathbb{E}_{\mathbb{S}(\phi_{\frac k{cN}})}c_{\vec{m}}\right)\right)\right|
\\\leq \frac{C}{\min\{N,(1+|k-l|)^2\}}\leq \frac{C}{1+|k-l|}
\end{multline*}
for all $k,l\in\mathcal{I}_N(\varepsilon_0)$. Summing up over all such $k$ and $l$, we obtain
% \begin{equation*}
% \mathbb{E}_{\mathbb{P}_N^n}\left|
% \frac 1{cN} \sum_{k\in\mathcal{I}_N(\varepsilon_0)}f\left(\frac k{cN}\right)c_{k+\vec{m}}-\frac 1{cN} \sum_{k\in\mathcal{I}_N(\varepsilon_0)}f\left(\frac k{cN}\right)\mathbb{E}_{\mathbb{S}(\phi_{\frac k{cN}})}c_{\vec{m}}\right|^2
% \leq \frac{C N\ln(N)}{N^2}.
% \end{equation*}
\begin{multline*}
\mathbb{E}_{\mathbb{P}_N^n}\left|
\frac 1{cN} \sum_{k\in\mathcal{I}_N(\varepsilon_0)}f\left(\frac k{cN}\right)c_{k+\vec{m}}-\frac 1{cN} \sum_{k\in\mathcal{I}_N(\varepsilon_0)}f\left(\frac k{cN}\right)\mathbb{E}_{\mathbb{S}(\phi_{\frac k{cN}})}c_{\vec{m}}\right|^2
\\\leq \frac{C N\ln(N)}{N^2}.
\end{multline*}
Replacing the Riemann sum by the appropriate integral we obtain
\begin{equation}
\label{eq:sumCsBulk}
\lim_{\substack{N\rightarrow \infty\\n=\lfloor c^2N^2\rfloor}}\mathbb{E}_{\mathbb{P}_N^n}\left|
\frac 1{cN} \sum_{k\in\mathcal{I}_N(\varepsilon_0)}f\left(\frac k{cN}\right)c_{k+\vec{m}}-\int_{c-2+\varepsilon_0}^{c+2-\varepsilon_0}f(a)\mathbb{E}_{\mathbb{S}(\phi_a)}c_{\vec{m}}\ da\right|^2=0.
\end{equation}
It follows from Proposition \ref{prop:PnNNearEdges} that
\begin{equation*}
\lim_{\substack{N\rightarrow \infty\\n=\lfloor c^2N^2\rfloor}}\mathbb{E}_{\mathbb{P}_N^n}\left|\frac 1{cN} \sum_{\substack{k\notin\mathcal{I}^+_N(1,\delta)\\k\geq -N}}f\left(\frac k{cN}\right)c_{k+\vec{m}}(\lambda)
\right.\left.-\delta_{c<1}\frac 1{cN}\sum_{\substack{k\notin\mathcal{I}^+_N(1,\delta)\\-N\leq k\leq 0}}f\left(\frac k{cN}\right)\right|=0,
\end{equation*}
which, since $f$ is bounded, implies
\begin{equation}
\label{eq:sumCsTail}
\lim_{\substack{N\rightarrow \infty\\n=\lfloor c^2N^2\rfloor}}\mathbb{E}_{\mathbb{P}_N^n}\left|\frac 1{cN} \sum_{\substack{k\in\mathcal{I}_N(0)\\k\geq -N}}f\left(\frac k{cN}\right)c_{k+\vec{m}}(\lambda)-\delta_{c<1}\int_{-\frac 1c}^{c-2}f(a)\ da\right|=0.
\end{equation}
Combining \eqref{eq:sumCsBulk} and \eqref{eq:sumCsTail}, and taking the limit $\varepsilon_0\rightarrow 0$ completes the proof.

$\qed$
\end{proof}

\begin{corollary}
\label{cor:limitForHooks}
For any $\varepsilon>0$ we have
\begin{multline*}
\lim_{\substack{N\rightarrow \infty\\n=\lfloor c^2N^2\rfloor}}\mathbb{P}_N^n\left\{\lambda:\left|\hat{\theta}(\lambda)-\sum_{k=1}^{\infty}\left(\mathfrak{m}(k)\int_{c-2}^{c+2}\mathbb{E}_{\mathbb{S}(\phi_a)}c_{\{0\}}-\mathbb{E}_{\mathbb{S}(\phi_a)}c_{\{0,k\}} da\right)\right|<\varepsilon\right\}
\\=1,
\end{multline*}
where $\hat{\theta}(\lambda)$ and $\mathfrak{m}(k)$ are as in Proposition \ref{prop:intFormMeas}.
\end{corollary}
\begin{proof}
Given a Young diagram $\lambda$ and a positive integer $k$, let $h_k(\lambda)$ be the number of cells in $\lambda$ with hook length $k$. Since $h_k(\lambda)$ is equal to the number of pairs $(i,i-k)$ such that $c_i(\lambda)=1$ and $c_{i-k}(\lambda)=0$, we have
\begin{equation*}
h_k(\lambda)=\sum_{i=-\infty}^{\infty}(c_i(\lambda)-c_i(\lambda)c_{i-k}(\lambda)).
\end{equation*}
Applying Lemma \ref{lem:sumsOfCs}, for any $k\in\mathbb{N}$ and any $\varepsilon>0$ we obtain
\begin{equation}
\label{eq:limitForHookCounts}
\lim_{\substack{N\rightarrow \infty\\n=\lfloor c^2N^2\rfloor}}\mathbb{P}_N^n\left\{\lambda:\left|\frac {h_k(\lambda)}{cN}-\int_{c-2}^{c+2}\mathbb{E}_{\mathbb{S}(\phi_a)}c_{\{0\}}-\mathbb{E}_{\mathbb{S}(\phi_a)}c_{\{0,k\}}\ da\right|<\varepsilon\right\}=1.
\end{equation}
Notice that
\begin{equation*}
\hat{\theta}(\lambda)=\sum_{k=1}^{\infty}\frac{h_k(\lambda)}{\sqrt{n}}\mathfrak{m}(k).
\end{equation*}
Since each row of $\lambda$ can have at most one cell with hook length $k$ we have $h_k(\lambda)<N$, whence the expression 
\begin{equation*}
\left|\frac {h_k(\lambda)}{cN}-\int_{c-2}^{c+2}\left(\mathbb{E}_{\mathbb{S}(\phi_a)}c_{\{0\}}-\mathbb{E}_{\mathbb{S}(\phi_a)}c_{\{0,k\}}\right)\ da\right|
\end{equation*}
is bounded. Since the series $\sum_{k=1}^{\infty}\mathfrak{m}(k)$ is convergent, summing \eqref{eq:limitForHookCounts} in $k$ we obtain the statement of the corollary.

$\qed$
\end{proof}

Define 
\begin{equation*}
F_\lambda(x)=\sqrt{n}f_\lambda\left(\frac x{\sqrt{n}}\right)=\mathcal{L}_\lambda(x)-\sqrt{n}\Omega_c\left(\frac x{\sqrt{n}}\right).
\end{equation*}
We have
\begin{equation*}
\frac{\sqrt{n}}8\|f_\lambda\|_{\frac 12}^2=\frac 1{4\sqrt{n}}\int_0^\infty\int_{-\infty}^\infty\left(\frac{F_\lambda(t+h)-F_\lambda(t)}h\right)^2\ dt\ dh.
\end{equation*}

\begin{corollary}
\label{cor:limitForSobNormMain}
For any $h_0>0$ and for any $\varepsilon>0$ we have
\begin{multline}
\label{eq:SobNormMain}
\lim_{\substack{N\rightarrow \infty\\n=\lfloor c^2N^2\rfloor}}\mathbb{P}_N^n\left\{\lambda:\left|\frac 1{4\sqrt{n}}\int\limits_0^{h_0}\int\limits_{-\infty}^\infty\left(\frac{F_\lambda(t+h)-F_\lambda(t)}h\right)^2\ dt\ dh-\tilde{H}_c(h_0) \right|<\varepsilon\right\}\\=1,
\end{multline}
where
% \begin{equation*}
% \tilde{H}_c(h_0)=\frac 14\int_{c-2}^{c+2}\int_0^1\int_0^{h_0}\mathbb{E}_{\mathbb{S}(\phi_a)}\left(\frac{\mathcal{L}_\lambda(s+h)-\mathcal{L}_\lambda(s)}h\right.\left.-\frac 2\pi \arcsin\left(\frac{c+a}{2\sqrt{1+ac}}\right)\right)^2\ dh\ ds\ da.
% \end{equation*}
\begin{multline*}
\tilde{H}_c(h_0)=\frac 14\int_{c-2}^{c+2}\int_0^1\int_0^{h_0}\mathbb{E}_{\mathbb{S}(\phi_a)}\left(\frac{\mathcal{L}_\lambda(s+h)-\mathcal{L}_\lambda(s)}h\right.\\\left.-\frac 2\pi \arcsin\left(\frac{c+a}{2\sqrt{1+ac}}\right)\right)^2\ dh\ ds\ da.
\end{multline*}
\end{corollary}
\begin{proof}
For any $t$ and any $h$ such that $0<h\leq h_0$, we have
\begin{equation*}
\left|\frac{cN}{h}\left(\Omega_c\left(\frac{t+h}{cN}\right)-\Omega_c\left(\frac{t}{cN}\right)\right)-\Omega'_c\left(\frac t{cN}\right)\right|\leq \frac{C(h_0)}{cN}.
\end{equation*}
From \eqref{eq:Omega'} it follows that the integral in \eqref{eq:SobNormMain} is equal to the expression
\begin{equation}
\label{eq:sumDiffL-OmP}
\frac 1{4cN}\int_0^1\int_0^{h_0}\sum_{k=-N}^\infty\left(\frac{\mathcal{L}_\lambda(s+k+h)-\mathcal{L}_\lambda(s+k)}h-\Omega'_c\left(\frac{s+k}{cN}\right)\right)^2\ dh\ ds
\end{equation}
up to $o(1)$.
From the definition of $c_k(\lambda)$ (see Section \ref{sec:PointProc}) it follows that we can write
\begin{equation*}
\frac{\mathcal{L}_\lambda(s+k+h)-\mathcal{L}_\lambda(s+k)}h=1-\frac{2(1-s)}{h}c_k(\lambda)-\sum_{i=1}^{h-1}\frac{2}hc_{k+i}(\lambda)-\frac{2s}hc_{k+h}(\lambda),
\end{equation*}
which implies that the expression in \eqref{eq:sumDiffL-OmP} can be written in the form 
\begin{equation*}
\frac 1{cN}\sum_{\vec{m}\in I_{h_0}}\sum_k f\left(\frac k{cN}\right)c_{k+\vec{m}}
\end{equation*}
for some finite set $I_{h_0}$.
Thus, we can apply Lemma \ref{lem:sumsOfCs} to \eqref{eq:sumDiffL-OmP}, and obtain the corollary (it is easy to check that the contributions coming from the term $\delta_{c<1}\int f(a)\ da$ in Lemma \ref{lem:sumsOfCs} cancel out). 

$\qed$
\end{proof}

\begin{lemma}[The tail estimate]
\label{lem:limitForSobNormTail}
For any $\varepsilon>0$ there exists $h_0>0$ such that 
\begin{equation*}
\lim_{\substack{N\rightarrow \infty\\n=\lfloor c^2N^2\rfloor}}\mathbb{P}_N^n\left\{\lambda:\frac 1{4\sqrt{n}}\int_{h_0}^{\infty}\int_{-\infty}^\infty\left(\frac{F_\lambda(t+h)-F_\lambda(t)}h\right)^2\ dt\ dh<\varepsilon\right\}=1.
\end{equation*}
\end{lemma}
The proof of Lemma \ref{lem:limitForSobNormTail} is given in Section \ref{sec:Tail}. We now prove Theorem \ref{thm:main}.

\begin{proof}[of Theorem \ref{thm:main}]
It follows from Corollary \ref{cor:PnNNearEdges} that
\begin{equation*}
\lim_{\substack{N\rightarrow \infty\\n=\lfloor c^2N^2\rfloor}}\mathbb{P}_N^n\left\{\lambda:\left|\frac{\sqrt{n}}{2}\int_{|x-c|>2}G_c (x)f_\lambda(x)\ dx\right|<\varepsilon\right\}=1
\end{equation*}
for any $\varepsilon>0$. The theorem follows immediately from Proposition \ref{prop:intFormMeas}, Corollaries \ref{cor:limitForHooks} and \ref{cor:limitForSobNormMain}, and Lemmas \ref{lem:limitForContents} and \ref{lem:limitForSobNormTail}. For the constant $H_c$ we obtain the following formula:
\begin{multline}
\label{eq:Hc}
H_c=\sum_{k=1}^{\infty}\left(\mathfrak{m}(k)\int_{c-2}^{c+2}\mathbb{E}_{\mathbb{S}(\phi_a)}c_{\{0\}}-\mathbb{E}_{\mathbb{S}(\phi_a)}c_{\{0,k\}}\ da\right)
\\+\frac 14\int\limits_{c-2}^{c+2}\int\limits_0^1\int\limits_0^{\infty}\mathbb{E}_{\mathbb{S}(\phi_a)}\left(\frac{\mathcal{L}_\lambda(s+h)-\mathcal{L}_\lambda(s)}h-\frac 2\pi \arcsin\left(\frac{c+a}{2\sqrt{1+ac}}\right)\right)^2dh\ ds\ da.
\end{multline}

$\qed$
\end{proof}

\section{The tail estimate}
\label{sec:Tail}
The goal of this section is to prove Lemma \ref{lem:limitForSobNormTail}. To simplify notation, in this section we set $n=c^2N^2$.

For $\delta>0$ and $K>0$ denote
% \begin{equation*}
% \mathbb{Y}^n_N(K,\delta)=
% \left\{\lambda\in\mathbb{Y}^n_N:\supp F_\lambda(x)\subset \mathcal{I}^+_N(K,\delta) \right\},
% \end{equation*}
% \begin{multline*}
% \mathbb{Y}^n_N(K,\delta)=
% \left\{\lambda\in\mathbb{Y}^n_N:
% \right.\\\left.\supp F_\lambda(x)\subset \left((c-2)cN-Kc N^\delta, (c+2)cN+Kc N^\delta\right)\right\},
% \end{multline*}
% \begin{multline*}
% \mathbb{Y}^n_N(K,\delta)=
% \{\lambda\in\mathbb{Y}^n_N:\lambda_1\leq (c+2)cN+Kc N^\delta\}
% \\\cap\left\{\lambda\in\mathbb{Y}^n_N:\begin{cases}
% l(\lambda)\leq -(c-2)cN+Kc N^\delta,&c<1
% \\\lambda_N\geq N+(c-2)cN-Kc N^\delta,&c>1
% \end{cases}\right\}.
% \end{multline*}
\begin{equation*}
F_\lambda^{K,\delta}(k)=
\begin{cases}
F_\lambda(k),&k\in\mathcal{I}^-_N(K,\delta)
\\0,&\text{ otherwise }
\end{cases}
\end{equation*}
and
\begin{equation*}
F_\lambda(k,l)=\left(\frac{F_\lambda(k+l)-F_\lambda(k)}l\right)^2.
\end{equation*}

Notice that $F_\lambda(x)$ is a Lipschitz function with Lipschitz constant $2$. It was proven in \cite{Bu} that for a Lipschitz function with Lipschitz constant $2$ the truncated integral in the $\frac 12$--Sobolev norm can be approximated by a sum of the integrand. More precisely, Lemma 6.1 in \cite{Bu} implies:
\begin{lemma}
For any $\delta\in(0,\frac 12)$, any $K>0$, $L>0$ and any $\varepsilon>0$, there exists a number $h_0>1$ depending only on $\delta,K,L,\varepsilon$ and such that for all $h>h_0$, all $N\in\mathbb{N}$, $n= c^2N^2$, and all $\lambda\in\mathbb{Y}^n_N(K,\delta)$ we have the inequality
% \begin{equation*}
% %\label{eq:tailInt-Sum}
% \frac 1{4\sqrt{n}}\int_{h}^\infty\int_{-\infty}^\infty\left(\frac{F_\lambda(t+h)-F_\lambda(t)}h\right)^2\ dt\ dh
% \leq \frac 1{\sqrt{n}}\sum_{l=h}^\infty\sum_{k=-\infty}^\infty \left(\frac{F_\lambda^{L,\delta}(k+l)-F_\lambda^{L,\delta}(k)}l\right)^2+\varepsilon.
% \end{equation*}
\begin{multline*}
%\label{eq:tailInt-Sum}
\frac 1{4\sqrt{n}}\int_{h}^\infty\int_{-\infty}^\infty\left(\frac{F_\lambda(t+h)-F_\lambda(t)}h\right)^2\ dt\ dh
\\\leq \frac 1{\sqrt{n}}\sum_{l=h}^\infty\sum_{k=-\infty}^\infty \left(\frac{F_\lambda^{L,\delta}(k+l)-F_\lambda^{L,\delta}(k)}l\right)^2+\varepsilon.
\end{multline*}
\end{lemma}

% We will show that the expected value of the left-hand side of \eqref{eq:tailInt-Sum} converges to zero with respect to the measure $\mathbb{P}_N^n$. For that purpose we will first show that the expected value of that quantity with respect to the poissonization of $\mathbb{P}_N^n$, i.e. with respect to $\mathbb{P}_N^{\gamma^+,0}$ converges to zero, and then depoissonize.

We now prove Lemma \ref{lem:limitForSobNormTail}. 

\begin{proof}[of Lemma \ref{lem:limitForSobNormTail}.]
Fix $L>0$ and $\delta\in(\frac 13,\frac 12)$. It follows from Corollary \ref{cor:PnNNearEdges} that we can restrict to the Young diagrams in the set $\mathbb{Y}^n_N(L,\delta)$. Separating the terms where $F_\lambda^{L,\delta}(k+l)=0$ or $F_\lambda^{L,\delta}(k)=0$, we obtain
\begin{multline}
\label{eq:tailSumBulk+Far}
\frac 1{cN}\sum_{l=h}^\infty\sum_{k=-\infty}^\infty \left(\frac{F_\lambda^{L,\delta}(k+l)-F_\lambda^{L,\delta}(k)}l\right)^2\\\leq
\frac 1{cN}\sum_{\substack{k,k+l\in\mathcal{I}^-_N(L,\delta)\\l\geq h}}F_\lambda(k,l)
+\frac 2{cN}\sum_{k\in\mathcal{I}^-_N(L,\delta)}\frac{F_\lambda(k)^2}{\dist(k,\mathbb{Z}\backslash\mathcal{I}^-_N(L,\delta))}.
\end{multline}
% \begin{equation}
% \label{eq:tailSumBulk+Far}
% \frac 1{cN}\sum_{\substack{k,k+l\in\mathcal{I}^-_N(L,\delta)\\l\geq h}}F_\lambda(k,l)
% +\frac 2{cN}\sum_{k\in\mathcal{I}^-_N(L,\delta)}\frac{F_\lambda(k)^2}{\max\{h,\dist(k,\mathbb{Z}\backslash\mathcal{I}^-_N(L,\delta))\}}.
% \end{equation}

%We now estimate the first sum on the left-hand side of \eqref{eq:tailSumBulk+Far}. 
It is easy to see that if $k\in\mathcal{I}^-_N(L,\delta)$, then
\begin{equation*}
F_\lambda^{L,\delta}(k+1)-F_\lambda^{L,\delta}(k)=1-2c_k(\lambda)-cN\left(\Omega_c\left(\frac {k+1}{cN}\right)-\Omega_c\left(\frac {k}{cN}\right)\right).
\end{equation*}
Using Theorem \ref{thm:intFormK} and \eqref{eq:phi} we obtain
\begin{multline}
\label{eq:diffF}
F_\lambda^{L,\delta}(k+1)-F_\lambda^{L,\delta}(k)=2\left(\mathbb{E}_{\mathbb{P}^{\gamma^+,0}_N}c_k-c_k(\lambda)\right)+2\left(\frac {\phi_{\frac k{cN}}}{\pi}-K_{N,\gamma^+}(k,k)\right)
\\+\left(\frac 2\pi \arcsin\left(\frac{c+\frac k{cN}}{2\sqrt{1+\frac k{N}}}\right)-cN\left(\Omega_c\left(\frac {k+1}{cN}\right)-\Omega_c\left(\frac {k}{cN}\right)\right)\right).
\end{multline}
Since $\Omega_c'(x)$ is given by \eqref{eq:Omega'} and
\begin{equation*}
\Omega_c''(x)=\frac{2-c^2+cx}{2(1+cx)\sqrt{4-(x-c)^2}},
\end{equation*}
from the second degree Taylor polynomial approximation of $\Omega_c$ it follows that there exists a constant $C>0$ such that
% \begin{equation}
% \label{eq:diffOmegaOmega'}
% \left|\frac 2\pi \arcsin\left(\frac{c+\frac k{cN}}{2\sqrt{1+\frac k{N}}}\right)-cN\left(\Omega_c\left(\frac {k+1}{cN}\right)-\Omega_c\left(\frac {k}{cN}\right)\right)\right|
% \leq \frac{C}{\sqrt{4c^2N^2-(k-c^2N)^2}}
% \end{equation}
\begin{multline}
\label{eq:diffOmegaOmega'}
\left|\frac 2\pi \arcsin\left(\frac{c+\frac k{cN}}{2\sqrt{1+\frac k{N}}}\right)-cN\left(\Omega_c\left(\frac {k+1}{cN}\right)-\Omega_c\left(\frac {k}{cN}\right)\right)\right|
\\\leq \frac{C}{\sqrt{4c^2N^2-(k-c^2N)^2}}
\end{multline}
for all $k\in\mathcal{I}^-_N(L,\delta)$.

It follows from Lemma \ref{lem:depoisK2} that there exist constants $C_1,C_2>0$ such that for all $k\in\mathcal{I}^-_N(L,\delta)$ and for all $\gamma^+$ we have
\begin{equation}
\label{eq:diffKPhi}
\left|\frac {\phi_{\frac k{cN}}}{\pi}-K_{N,\gamma^+}(k,k)\right|\leq\frac{C_1 e^{C_2|\gamma^+-c^2N|}}{2cN-|k-c^2N|}.
\end{equation}
Since $\mathbb{E}_{\mathbb{P}^{\gamma^+,0}_N}(c_k)=K_{N,\gamma^+}(k,k)$, combining \eqref{eq:diffF}, \eqref{eq:diffOmegaOmega'} and \eqref{eq:diffKPhi} we obtain
\begin{equation}
\label{eq:diffFlVar+}
F_\lambda(k,l)
\\\leq \frac 2{l^2}\Var_{\mathbb{P}^{\gamma^+,0}_N}\left(c_k+\dots+c_{k+l-1}\right)+\frac 1{l^2}\left(\sum_{j=k}^{k+l-1}\frac{C_1e^{C_2|\gamma^+-c^2N|}}{2cN-|j-c^2N|}\right)^2,
\end{equation}
for some constants $C_1,C_2>0$ and for all $k$ and $l$ such that $k,k+l\in\mathcal{I}^-_N(L,\delta)$.

Summing the second term on the left--hand side of \eqref{eq:diffFlVar+} we obtain
\begin{multline*}
\frac 1{cN}\sum_{\substack{k,k+l\in\mathcal{I}^-_N(L,\delta)\\l\geq h}}\frac 1{l^2}\left(\sum_{j=k}^{k+l-1}\frac{1}{2cN-|j-c^2N|}\right)^2
\leq\frac 1{cN}\sum_{l=h}^{4cN}\frac 1{l^2}\sum_{k=1}^{4cN}\left(\sum_{j=k}^{k+l}\frac 1j\right)^2
\\\leq\frac{2}{cN}\sum_{l=h}^{4cN}\frac 1{l^2}\sum_{k=1}^{4cN}\left(\ln(k+l)-\ln(k)\right)^2
\leq 8\sum_{l=h}^{4cN}\frac {(\ln l)^2}{l^2}\leq 20\frac {(\ln h)^2}{h}.
\end{multline*}
Combining this with the estimate of the variance given in Lemma \ref{lem:variancePoisB} below, we obtain
\begin{equation}
\label{eq:EstTailMain}
\frac 1{cN}\sum_{\substack{k,k+l\in\mathcal{I}^-_N(L,\delta)\\l\geq h}}F_\lambda(k,l)
\leq C_1 e^{C_2|\gamma^+-c^2N|}\frac{(\ln h)^2}h.
\end{equation}

We now turn to estimating the second sum on the right--hand side of \eqref{eq:tailSumBulk+Far}. We will estimate the sum when $k$ is in the left half of the interval $\mathcal{I}^-_N(L,\delta)$, i.e. when $k$ is larger than $c^2N$. The sum when $k<c^2N$ can be estimated completely similarly. Denote 
\begin{equation*}
M_N(L,\delta)=\max\{M:M\in\mathcal{I}^-_N(L,\delta)\}.
\end{equation*}
We have 
\begin{equation*}
M_N(L,\delta)=2cN-LcN^\delta+c^2N
\end{equation*}
and
\begin{equation*}
\dist(k,\mathbb{Z}\backslash\mathcal{I}^-_N(L,\delta))=M_N(L,\delta)-k.
\end{equation*}
Notice that
\begin{equation}
\label{eq:FasLastandDiff}
F_\lambda(k)^2\leq 2(F_\lambda(k)-F_\lambda(M_N(L,\delta)))^2+2F_\lambda(M_N(L,\delta))^2.
\end{equation}
Since $F_\lambda$ is Lipschitz with constant $2$ and $\lambda\in\mathbb{Y}^n_N(L,\delta)$, we have
\begin{equation*}
|F_\lambda(M_N(L,\delta))|\leq 4LN^\delta,
\end{equation*}
which implies
\begin{equation}
\label{eq:FLastSum}
\frac 2{cN}\sum_{k\in\mathcal{I}^-_N(L,\delta)}\frac{F_\lambda(M_N(L,\delta))^2}{M_N(L,\delta)-k}\leq \frac {32L^2N^{2\delta}\ln(4cN)}{cN}.
\end{equation}
Since
% \begin{equation*}
% \frac 2{cN}\sum_{\substack{k\in\mathcal{I}^-_N(L,\delta)\\k>c^2N}}\frac {\left(\sum_{j=k}^{M_N(L,\delta)}\frac{1}{2cN-|j-c^2N|}\right)^2}{{M_N(L,\delta)-k}}
% \leq \frac {2}{cN}\sum_{\substack{k\in\mathcal{I}^-_N(L,\delta)\\k>c^2N}}\frac {(\ln(M_N(L,\delta)-k))^2}{M_N(L,\delta)-k}\leq\frac{2(\ln(4cN))^3}{cN},
% \end{equation*}
\begin{multline*}
\frac 2{cN}\sum_{\substack{k\in\mathcal{I}^-_N(L,\delta)\\k>c^2N}}\frac {\left(\sum_{j=k}^{M_N(L,\delta)}\frac{1}{2cN-|j-c^2N|}\right)^2}{{M_N(L,\delta)-k}}
\\\leq \frac {2}{cN}\sum_{\substack{k\in\mathcal{I}^-_N(L,\delta)\\k>c^2N}}\frac {(\ln(M_N(L,\delta)-k))^2}{M_N(L,\delta)-k}\leq\frac{2(\ln(4cN))^3}{cN},
\end{multline*}
it follows from \eqref{eq:diffFlVar+}, \eqref{eq:FasLastandDiff} and \eqref{eq:FLastSum} that
\begin{multline*}
\frac 2{cN}\sum_{\substack{k\in\mathcal{I}^-_N(L,\delta)\\k>c^2N}}\frac {F_\lambda(k)^2}{M_N(L,\delta)-k}
\\\leq \frac 4{cN}\sum_{\substack{k\in\mathcal{I}^-_N(L,\delta)\\k>c^2N}}\frac{\Var_{\mathbb{P}^{\gamma^+,0}_N}\left(c_k+\dots+c_{M_N(L,\delta)}\right)}{M_N(L,\delta)-k}+\frac {C_1e^{C_2|\gamma^+-c^2N|}\ln(4cN)}{N^{1-2\delta}}.
\end{multline*}
Using the estimate of the variance given in Lemma \ref{lem:variancePoisE} below, we obtain
\begin{equation*}
\frac 2{cN}\sum_{\substack{k\in\mathcal{I}^-_N(L,\delta)\\k>c^2N}}\frac {F_\lambda(k)^2}{M_N(L,\delta)-k}
\leq \frac{C_1e^{C_2|\gamma^+-c^2N|}}{N^\frac 16}.
\end{equation*}
Combining this with \eqref{eq:EstTailMain} we obtain
% \begin{equation*}
% \mathbb{E}_{\mathbb{P}^{\gamma^+,0}_N}\left(\frac 1{\sqrt{n}}\sum_{l=h}^\infty\sum_{k=-\infty}^\infty \left(\frac{F_\lambda^{L,\delta}(k+l)-F_\lambda^{L,\delta}(k)}l\right)^2\right)
% \leq C_1e^{C_2|\gamma^+-c^2N|}\left(\frac 1{N^\frac 16}+\frac{(\ln h)^2}h\right),
% \end{equation*}
\begin{multline*}
\mathbb{E}_{\mathbb{P}^{\gamma^+,0}_N}\left(\frac 1{\sqrt{n}}\sum_{l=h}^\infty\sum_{k=-\infty}^\infty \left(\frac{F_\lambda^{L,\delta}(k+l)-F_\lambda^{L,\delta}(k)}l\right)^2\right)
\\\leq C_1e^{C_2|\gamma^+-c^2N|}\left(\frac 1{N^\frac 16}+\frac{(\ln h)^2}h\right),
\end{multline*}
which implies the Lemma after depoissonization.

$\qed$
\end{proof}

\begin{lemma}
\label{lem:variancePoisB}
Let
\begin{equation*}
V_{\gamma^+,N}^{L,\delta}(h)=\frac 1N \sum_{\substack{k,k+l\in\mathcal{I}^-_N(L,\delta)\\l\geq h}}\frac 1{l^2}\Var_{\mathbb{P}_N^{\gamma^+,0}}\left(c_k+\ldots+c_{k+l-1}\right).
\end{equation*}
For any $\delta>\frac 13$ and $L>0$ there exist constants $C_1>0$ and $C_2>0$ such that for any $h>0$ there exists $N_0$ such that for all $N>N_0$ and all $\gamma^+$ we have
\begin{equation*}
V_{\gamma^+,N}^{L,\delta}(h)\leq C_1 e^{C_2|\gamma^+-c^2N|}\frac{\ln h}h.
\end{equation*}
\end{lemma}

\begin{proof}
We can assume $h<N^{\delta}$. Throughout the proof, $C_1$ and $C_2$ will denote arbitrary constants that depend only on $L$ and $\delta$.
It is immediate from \eqref{eq:KProj} that 
\begin{equation}
\label{eq:VarSumKK}
\Var_{\mathbb{P}_N^{\gamma^+,0}}\left(c_k+\ldots+c_{k+l-1}\right)=\sum_{x\in[k,k+l-1]}\sum_{y\notin[k,k+l-1]}K_{N,\gamma^+}(x,y)K_{N,\gamma^+}(y,x).
\end{equation}
Summing over $k,k+l\in\mathcal{I}^-_N(L,\delta)$, $l\geq h$, we obtain
\begin{equation*}
V_{\gamma^+,N}^{L,\delta}(h)=\frac 1N\sum_{\substack{k,k+l\in\mathcal{I}^-_N(L,\delta)\\l\geq h}}\sum_{x\in[k,k+l-1]}\sum_{y\notin[k,k+l-1]}\frac{K_{N,\gamma^+}(x,y)K_{N,\gamma^+}(y,x)}{l^2}.
\end{equation*}
Let $P^h_N(x,y)$ be the coefficient of $K_{N,\gamma^+}(x,y)K_{N,\gamma^+}(y,x)$ in the above sum and let 
\begin{equation*}
Q^h_{N,\gamma^+}(x,y)=P^h_N(x,y)K_{N,\gamma^+}(x,y)K_{N,\gamma^+}(y,x).
\end{equation*}
We have
\begin{equation*}
V_{\gamma^+,N}^{L,\delta}(h)=\frac 1N\sum_{\substack{x\in\mathcal{I}^-_N(L,\delta)\\y\in\mathbb{Z}}}Q^h_{N,\gamma^+}(x,y).
\end{equation*}

When $x\in\mathcal{I}^-_N(L,\delta)$ and $y\in\mathcal{I}^+_N(L,\delta)$, estimating from above the number of intervals of length $l\geq h$ that contain $x$ but not $y$, we obtain
\begin{equation*}
P^h_N(x,y)\leq 2\sum_{l=h}^\infty \frac {\min\{|x-y|,l\}}{l^2}\leq \psi(h,|x-y|),
\end{equation*}
where
\begin{equation*}
\psi(h,l)=\begin{cases}
\frac{2l}{h-1},&l\leq h\\
4\ln l,&l>h
\end{cases}.
\end{equation*}

Since for a fixed $l$ the number of pairs $x,y\in\mathcal{I}^-_N(L,\delta)$ such that $|x-y|=l$ is less than $4cN$, it follows from Remark \ref{rem:KsqEdge2} that
\begin{equation}
\label{eq:VEInsideSame}
\frac 1N\sum_{\substack{x\in \mathcal{I}^-_N(L,\delta)\\y\in\mathcal{I}^-_N(\frac L 2,\delta)\\xy>0}} |Q^h_{N,\gamma^+}(x,y)|
\leq \sum_{l=1}^\infty \frac{\psi(h,l) C_1e^{C_2|\gamma^+-c^2N|}}{(1+l)^2}
\leq C_1e^{C_2|\gamma^+-c^2N|}\frac{\ln h}h.
\end{equation}

Similarly, it follows from Lemma \ref{lem:estKDiffDelta} with $\delta_1=\delta_2=1$ that for any $\varepsilon>0$, 
\begin{equation}
\label{eq:VEBulk}
\frac 1N\sum_{\substack{x,y\in \mathcal{I}_N(\varepsilon)}} |Q^h_{N,\gamma^+}(x,y)|\leq C_1e^{C_2|\gamma^+-c^2N|}\frac{\ln h}h.
\end{equation}

If $x\in \mathcal{I}^-_N(L,\delta)$, $y\in\mathcal{I}^-_N(\frac L 2,\delta)\backslash\mathcal{I}_N(\varepsilon)$, and $x$ and $y$ have opposite signs, then $\frac {2cN}3\leq|x-y|\leq 4cN$, whence Lemma \ref{lem:estKDiffDelta} implies
\begin{equation*}
|Q^h_{N,\gamma^+}(x,y)|\leq C_1e^{C_2|\gamma^+-c^2N|}\frac{N^{1-\delta}\ln N}{N^2}.
\end{equation*}
Since the cardinality of the set $\mathcal{I}^-_N(L,\delta)\times(\mathcal{I}^-_N(\frac L 2,\delta)\backslash\mathcal{I}_N(\varepsilon))$ is less than $16c^2N^2$, we obtain
\begin{equation}
\label{eq:VEInsideDiff}
\frac 1N\sum_{\substack{x\in \mathcal{I}^-_N(L,\delta)\\y\in\mathcal{I}^-_N(\frac L 2,\delta)\backslash\mathcal{I}_N(\varepsilon)\\xy<0}} |Q^h_{N,\gamma^+}(x,y)|
\leq C_1e^{C_2|\gamma^+-c^2N|}\frac{\ln N}{N^\delta}.
\end{equation}

If $x\in \mathcal{I}^-_N(L,\delta)$ and $y\in\mathcal{I}^+_N(\frac L 2,\delta)\backslash\mathcal{I}^-_N(\frac L 2,\delta)$, then $\frac 12 cLN^\delta<|x-y|< 5cN$, whence Lemma \ref{lem:estKBulkEdge} implies
\begin{equation*}
|Q^h_{N,\gamma^+}(x,y)|\leq C_1e^{C_2|\gamma^+-c^2N|}\frac{N^{\frac{5-3\delta}6}\ln N}{(1+|x-y|)^2}.
\end{equation*}
Since the cardinality of $\mathcal{I}^+_N(\frac L 2,\delta)\backslash\mathcal{I}^-_N(\frac L 2,\delta)$ is less than $LcN^\delta$ and for a fixed $$y\in\mathcal{I}^+_N(\frac L 2,\delta)\backslash\mathcal{I}^-_N(\frac L 2,\delta)$$ we have 
\begin{equation*}
\displaystyle\sum_{x\in \mathcal{I}^-_N(L,\delta)}\frac 1{(1+|x-y|)^2}\leq \frac 2{LcN^{\delta}},
\end{equation*}
we obtain
\begin{equation}
\label{eq:VEEdge}
\frac 1N\sum_{\substack{x\in \mathcal{I}^-_N(L,\delta)\\y\in\mathcal{I}^+_N(\frac L 2,\delta)\backslash\mathcal{I}^-_N(\frac L 2,\delta)}} |Q^h_{N,\gamma^+}(x,y)|
\leq C_1e^{C_2|\gamma^+-c^2N|}\frac{\ln N}{N^{\frac 16+\frac \delta 2}}.
\end{equation}

When $x\in\mathcal{I}^-_N(L,\delta)$ and $y\notin\mathcal{I}^+_N(L,\delta)$, summing over all subintervals of $\mathcal{I}^-_N(L,\delta)$ of length at least $h$, we obtain
\begin{equation*}
P^h_N(x,y)\leq\sum_{l=h}^{|\mathcal{I}^-_N(L,\delta)|}\frac{|\mathcal{I}^-_N(L,\delta)|-l}{l^2}\leq \frac N{h-1}.
\end{equation*}
Using Lemma \ref{lem:estKExtremes} to estimate $|K_{N,\gamma^+}(x,y)K_{N,\gamma^+}(y,x)|$, we obtain
\begin{equation}
\label{eq:VEExtremes}
\frac 1N\sum_{\substack{x\in \mathcal{I}^-_N(L,\delta)\\y\notin \mathcal{I}^+_N(L,\delta)}} |Q^h_{N,\gamma^+}(x,y)|
\leq C_1e^{C_2|\gamma^+-c^2N|}\frac{e^{-C_3N^{\frac {3\delta}2-\frac 12}}}{(1+|x-y|)^2}.
\end{equation}
Combining the estimates \eqref{eq:VEInsideSame}, \eqref{eq:VEBulk}, \eqref{eq:VEInsideDiff}, \eqref{eq:VEEdge} and \eqref{eq:VEExtremes} completes the proof.

$\qed$
\end{proof}

\begin{lemma}
\label{lem:variancePoisE}
For any $\delta>\frac 13$ and $L>0$ there exist constants $C_1>0$ and $C_2>0$ such that for any $h>0$ there exists $N_0$ such that for all $N>N_0$ and all $\gamma^+$ we have
\begin{equation*}
\frac 2{cN}\sum_{\substack{k\in\mathcal{I}^-_N(L,\delta)\\k>c^2N}}\frac{\Var_{\mathbb{P}^{\gamma^+,0}_N}\left(c_k+\dots+c_{M_N(L,\delta)}\right)}{M_N(L,\delta)-k}
\leq \frac{C_1e^{C_2|\gamma^+-c^2N|}}{N^{\frac 16}}.
\end{equation*}
\end{lemma}

\begin{proof}
Using \eqref{eq:VarSumKK} we can write the sum of the variance in the form
\begin{equation*}
\frac 2{cN}\sum_{\substack{k\in\mathcal{I}^-_N(L,\delta)\\k>c^2N}}\frac{\Var_{\mathbb{P}^{\gamma^+,0}_N}\left(c_k+\dots+c_{M_N(L,\delta)}\right)}{M_N(L,\delta)-k}=\frac 2{cN}\sum_{\substack{x\in\mathcal{I}^-_N(L,\delta)\\x>c^2N\\y\notin[x,M_N(L,\delta)]}}S^h_{N,\gamma^+}(x,y),
\end{equation*}
where
\begin{equation*}
S^h_{N,\gamma^+}(x,y)=R^h_N(x,y)K_{N,\gamma^+}(x,y)K_{N,\gamma^+}(y,x)
\end{equation*}
and
\begin{equation*}
R^h_N(x,y)=
\begin{cases}
\sum\limits_{k=y}^{x-1} \frac 1{M_N(L,\delta)-k},&y\in(c^2N,x)\\
\sum\limits_{k=c^2N}^{x-1} \frac 1{M_N(L,\delta)-k},&y<c^2N\text{ or }y>M_N(L,\delta)
\end{cases}.
\end{equation*}
Since
\begin{equation*}
R^h_N(x,y)\leq
\frac{|x-y|}{M_N(L,\delta)-x}\text{ if }y\in\mathcal{I}^+_N(L,\delta),
\end{equation*}
it follows from Lemma \ref{lem:estKBulkEdge} that
\begin{equation}
\label{eq:varPoisEM}
\frac 2{cN}\sum_{\substack{x\in\mathcal{I}^-_N(L,\delta)\\y\in\mathcal{I}^+_N(L,\delta)}}S^h_{N,\gamma^+}(x,y)
\leq C_1e^{C_2|\gamma^+-c^2N|}\frac{(\ln M_N(L,\delta))^2}{N^{\frac 16+\frac \delta 2}}
\leq \frac{C_1e^{C_2|\gamma^+-c^2N|}}{N^{\frac 16}}.
\end{equation}
Since
\begin{equation*}
R^h_N(x,y)\leq
\frac{4cN}{M_N(L,\delta)-x}\quad\text{ if }y\notin\mathcal{I}^+_N(L,\delta),
\end{equation*}
it follows from Lemma \ref{lem:estKExtremes} that
\begin{equation}
\label{eq:varPoisEE}
\frac 2{cN}\sum_{\substack{x\in\mathcal{I}^-_N(L,\delta)\\y\notin\mathcal{I}^+_N(L,\delta)}}S^h_{N,\gamma^+}(x,y)
\leq C_1e^{C_2|\gamma^+-c^2N|}e^{-C_3N^{\frac {3\delta}2-\frac 12}}.
\end{equation}
Combining the estimates \eqref{eq:varPoisEM} and \eqref{eq:varPoisEE} completes the proof.

$\qed$
\end{proof}

% \begin{lemma}
% \label{lem:variance}
% For any $\delta>\frac 13$ and for any $h>0$ there exists $N_0$ such that for all $L>0$, all $N>N_0$ and all $\gamma^+$ we have
% \begin{equation*}
% \frac 1N \sum_{\substack{k,k+l\in\mathcal{I}^-_N(L,\delta)\\l\geq h}}\frac 1{l^2}Var_{\mathbb{P}_N^n}\left(c_k+\ldots+c_{k+l-1}\right)
% \leq C_1\frac{\ln h}h e^{C_2|\gamma^+-c^2N|}.
% \end{equation*}
% \end{lemma}
% \begin{proof}
% This follows immediately from Lemma \ref{lem:variancePois} by depoissonizing.
% 
% $\qed$
% \end{proof}

\bibliography{mybib}

\begin{thebibliography}{BOO00}

\bibitem[Bia01]{Biane2001}
P.~Biane.
\newblock Approximate factorization and concentration for characters of
  symmetric groups.
\newblock {\em Int. Math. Res. Notices.}, 2001(4):179--192, 2001.

\bibitem[BK08]{BK}
A.~Borodin and J.~Kuan.
\newblock Asymptotics of {P}lancherel measures for the infinite-dimensional
  unitary group.
\newblock {\em Advances in Mathematics}, 219(3):894--931, 2008.

\bibitem[BO07]{BOPlancherelType}
A.~Borodin and G.~Olshanski.
\newblock Asymptotics of {P}lancherel-type random partitions.
\newblock {\em Journal of Algebra}, 313:40--60, 2007.

\bibitem[BO12]{BOBoundaryGT}
A.~Borodin and G.~Olshanski.
\newblock The boundary of the {G}elfand-{T}setlin graph: {A} new approach.
\newblock {\em Advances in Mathematics}, 230(4-6):1738--1779, 2012.

\bibitem[BOO00]{BOO}
A.~Borodin, A.~Okounkov, and G.~Olshanski.
\newblock Asymptotics of {P}lancherel measures for symmetric groups.
\newblock {\em J. Amer. Math. Soc.}, 13(3):481--515, 2000.

\bibitem[Buf10]{Bu}
A.~I. Bufetov.
\newblock On the {V}ershik-{K}erov conjecture concerning the
  {S}hannon-{M}acmillan-{B}reiman theorem for the {P}lancherel family of
  measures on the space of {Y}oung diagrams.
\newblock 2010.
\newblock to appear in Geometric and Functional Analysis. arXiv:1001.4275v1
  [math.RT].

\bibitem[FH91]{FultonHarris}
W.~Fulton and J.~Harris.
\newblock {\em Representation theory. A first course.}, volume 129 of {\em
  Graduate Texts in Mathematics}.
\newblock Springer-Verlag, New York, 1991.

\bibitem[Joh01]{JohDisc2001}
K.~Johansson.
\newblock Discrete orthogonal polynomial ensembles and the {P}lancherel
  measure.
\newblock {\em Annals of Mathematics}, 153:259--296, 2001.

\bibitem[LS77]{LoganShepp}
B.~F. Logan and L.~A. Shepp.
\newblock A variational problem for random {Y}oung tableaux.
\newblock {\em Acta Mathematica}, 26(2):206--222, 1977.

\bibitem[Mkr12]{M2}
S.~Mkrtchyan.
\newblock Asymptotics of the maximal and the typical dimensions of isotypic
  components of tensor representations of the symmetric group.
\newblock {\em European Journal of Combinatorics}, 33(7):1631--1652, 2012.
\newblock 10.1016/j.ejc.2012.03.023.

\bibitem[Oko02]{OkNASA2001}
A.~Okounkov.
\newblock Symmetric functions and random partitions.
\newblock In {\em Symmetric functions 2001: surveys of developments and
  perspectives}, volume~74 of {\em NATO Sci. Ser. II Math. Phys. Chem.}, pages
  223--252. Kluwer Acad. Publ., Dordrecht, 2002.

\bibitem[Ols08]{OlshDiffOpDetProc}
G.~Olshanski.
\newblock Difference operators and determinantal point processes.
\newblock {\em Funct. Anal. Appl.}, 42(4):317--329, 2008.
\newblock 10.1007/s10688-008-0045-z.

\bibitem[Ols09]{OlshNotes}
G.~Olshanski.
\newblock Asymptotic representation theory: {L}ectures at {I}ndependent
  {U}niversity of {M}oscow {I}{I}.
\newblock {\em Lecture Notes}, 2009.
\newblock http://www.iitp.ru/en/userpages/88/.

\bibitem[OO98]{OkOlJack}
A.~Okounkov and G.~Olshanski.
\newblock Asymptotics of {J}ack polynomials as the number of variables goes to
  infinity.
\newblock {\em Intern. Math. Res. Notices}, (13):641--682, 1998.

\bibitem[Sos00]{SoshDetP}
A.~Soshnikov.
\newblock Determinantal random point fields.
\newblock {\em Uspekhi Mat. Nauk}, 55(5):107--160, 2000.
\newblock English translation: Russian Math. Surveys 55(2000), no. 5, 923--975.

\bibitem[VK77]{VK77}
A.~M. Vershik and S.~V. Kerov.
\newblock Asymptotics of the {P}lancherel measure of the symmetric group.
\newblock {\em Soviet Math. Dokl.}, 18:527--531, 1977.

\bibitem[VK82]{VK82}
A.~M. Vershik and S.~V. Kerov.
\newblock Characters and factor representations of the infinite unitary group.
\newblock {\em Soviet Math. Doklady}, 26:570--574, 1982.

\bibitem[VK85]{VK85}
A.~M. Vershik and S.~V. Kerov.
\newblock Asymptotic behavior of the maximum and generic dimensions of
  irreducible representations of the symmetric group.
\newblock {\em Funktsional. Anal. i Prilozhen.}, 19(1):25--36, 1985.

\bibitem[Voi76]{Voi}
D.~Voiculescu.
\newblock Repr{\'e}sentations factorielles de type {I}{I}1 de {U}(infty).
\newblock {\em J. Math. Pures Appl.}, 55(1):1--20, 1976.

\bibitem[VP10]{VP}
A.~M. Vershik and D.~Pavlov.
\newblock Some numerical and algorithmical problems in the asymptotic
  representation theory.
\newblock 2010.
\newblock arXiv:1004.1869v1 [math.RT].

\bibitem[Wey39]{W}
H.~Weyl.
\newblock {\em The {C}lassical {G}roups: {T}heir {I}nvariants and
  {R}epresentations}.
\newblock Princeton University Press, Princeton, N.J., 1939.

\end{thebibliography}
\bibliographystyle{alpha}

\end{document}